\documentclass[reqno]{amsart}
\pdfoutput=1

\usepackage[utf8]{inputenc}
\usepackage[english]{babel}

\usepackage{mathtools}
\usepackage{xcolor}
\usepackage{tikz}
\usetikzlibrary{arrows.meta}
\usepackage{caption}
\usepackage{xcolor}
\usepackage{amssymb}
\usepackage{float}
\usepackage{amsaddr}
\usepackage{hyperref}

\usepackage[pagewise]{lineno}

\allowdisplaybreaks

\usepackage[shortlabels]{enumitem}

\usepackage{accents}

\theoremstyle{plain}
\newtheorem{theorem}{Theorem}

\newtheorem{lemma}[theorem]{Lemma}

\theoremstyle{definition}
\newtheorem{definition}[theorem]{Definition}
\newtheorem{remark}[theorem]{Remark}

\let\d\undefined
\let\H\undefined

\newcommand*{\N}{\mathbb{N}}

\newcommand*{\R}{\mathbb{R}}

\newcommand*{\d}{\mathrm{d}}
\newcommand{\der}{\mathrm{d}}
\newcommand*{\norm}[1]{\left\lVert#1\right\rVert}
\newcommand{\ip}[2]{\left\langle#1,#2\right\rangle}
\newcommand{\iip}[2]{\left(#1,#2\right)}
\newcommand{\H}{\mathcal{H}}
\newcommand{\V}{\mathcal{V}}

\newcommand{\eps}{\varepsilon}
\newcommand{\Order}{\mathcal{O}}

\newcommand{\dooin}{\partial_{\mathrm{in}}}
\newcommand{\dooout}{\partial_{\mathrm{out}}}

\newcommand{\grad}[1]{\overset{\mathtt{#1}}{\nabla}}
\newcommand{\dive}[1]{\overset{\mathtt{#1}}{\operatorname{div}}}
\newcommand{\abs}[1]{\left\vert#1\right\vert}

\newcommand{\alf}[1]{\accentset{\alpha}{#1}}
\newcommand{\bet}[1]{\accentset{\beta}{#1}}
\newcommand{\alfGrad}[1]{\grad{#1}_{\!\scriptscriptstyle\alpha}}
\newcommand{\alfDive}[1]{\dive{#1}_{\scriptscriptstyle\alpha}}

\newcommand{\SSob}{K^2}

\DeclareMathOperator{\Lip}{Lip}
\DeclareMathOperator{\adj}{adj}
\DeclareMathOperator{\sisus}{int}

\newcommand{\NTR}[1]{}

\title[Pestov identities and X-ray tomography in low regularity]{Pestov identities and X-ray tomography on manifolds of low regularity}
\author{Joonas Ilmavirta}
\address{Department of Mathematics and Statistics\\
University of Jyv\"askyl\"a\\
P.O. Box 35 (MaD)\\
FI-40014 University of Jyv\"askyl\"a, Finland\\
\texttt{joonas.ilmavirta@jyu.fi}}
\author{Antti Kykkänen}
\address{Department of Mathematics and Statistics\\
University of Helsinki\\
P.O. Box 68 (Gustaf H\"allstr\"omin katu 2B)\\
FI-00014 University of Helsinki, Finland\\
\texttt{antti.k.kykkanen@jyu.fi}}
\date{\today}
\keywords{Non-smooth geometry, X-ray tomography, integral geometry, inverse problems}
\subjclass[2010]{44A12, 53C22, 53C65, 58J32}

\begin{document}

\maketitle

\begin{abstract}
We prove that the geodesic X-ray transform is injective on scalar functions and (solenoidally) on one-forms on simple Riemannian manifolds $(M,g)$ with $g \in C^{1,1}$. In addition to a proof, we produce a redefinition of simplicity that is compatible with rough geometry. This $C^{1,1}$-regularity is optimal on the H\"older scale. The bulk of the article is devoted to setting up a calculus of differential and curvature operators on the unit sphere bundle atop this non-smooth structure.
\end{abstract}

\section{Introduction}

\NTR{We have indicated with these footnotes all changes made to the article based on the referees' suggestions.}
\NTR{Updated email.}

How regular does a Riemannian metric have to be for the geodesic X-ray transform to be injective?
It is well known (see e.g.~\cite{MukhometovRPTDRM,MukhometovOPRRM,RomanovIPMP,ARUDFFDIAG}) that on a smooth simple Riemannian manifold this injectivity property holds.
If the regularity is too low, the question itself falls apart:
If the Riemannian metric is~$C^{1,\alpha}$ for~$\alpha<1$, then the geodesic equation can fail to have unique solutions~\cite{HartmanLUG, SSGLR}.
Therefore it is indeed in a sense optimal on the H\"older scale when we prove that on a~$C^{1,1}$-smooth simple Riemannian manifold the geodesic X-ray transform is injective on scalars and one-forms, the latter one up to natural gauge.

The geodesic X-ray transform is ubiquitous in the theory of geometric inverse problems.
It appears either directly or through linearization in many imaging problems of anisotropic and inhomogeneous media.
Most inverse problems have been studied in smooth geometry but the nature is not smooth.
The irregularities of the structure of the Earth range from individual rocks (zero-dimensional, small) to interfaces like the core--mantle boundary (two-dimensional, global scale).
Irregularity across various scales and dimensions are most conveniently captured in a single geometric structure of minimal regularity assumptions.
Specific kinds of irregularities can well be analyzed further, but we restrict our attention to a uniform and global but low regularity.

We prove this injectivity result by using a Pestov identity, an approach that can well be called classical (cf.~\cite{MukhometovRPTDRM,MukhometovOPRRM,RomanovIPMP,ARUDFFDIAG,PSUTTPC,IMIGMBA,UhlmannIPSU,PSUGIPETD}).
What requires care is keeping track of regularity.
The manifold does not have natural structure beyond~$C^{1,1}$, so regularity beyond is both useless and inaccessible.
The natural differential operators on the manifold and its unit sphere bundle are not smooth, and only a couple of derivatives of any kind can be taken at all.
The various commutators that appear in the calculations have to be interpreted in a suitable way, so that~$[A,B]$ exists reasonably even when the products~$AB$ and~$BA$ do not.
We employ two methods around these obstacles: approximation by smooth structures and careful analysis in the non-smooth geometry.

We say that a function is in the class~$C^{1,1}$ if it is continuously differentiable and the derivative is Lipschitz, and we define in definition~\ref{def:simple-b} what a~$C^{1,1}$ simple Riemannian metric is.
\NTR{Replaced the first sentence.}
Throughout the article our manifolds are assumed to be connected and to have dimension~$n \ge 2$.

\begin{theorem}
\label{thm:c11-injectivity}
Let~$(M,g)$ be a simple~$C^{1,1}$ manifold in the sense of definition~\ref{def:simple-b}.
\begin{enumerate}[(1)]
    \item\label{injectivity-claim-1} If~$f$ is a Lipschitz function on~$M$ that integrates to zero over all maximal geodesics of~$M$, then~$f = 0$.
    \item\label{injectivity-claim-2} Let~$h$ be a Lipschitz~$1$-form on~$M$ that vanishes on the boundary $\partial M$.
    Then~$h$ integrates to zero over all maximal geodesics of~$M$ if and only if there is a scalar function~$p \in C^{1,1}(M)$ vanishing on the boundary~$\partial M$ so that~$h = \d p$.
\end{enumerate}
\end{theorem}

We have to redefine simplicity to be tractable in our rough setup, and we regard this new definition as one of our main results.
To verify that our redefinition is a valid one, we prove that it agrees with the classical definition when the metric is smooth.
The classical definition of a smooth simple manifold implies the existence of global coordinates, but in the~$C^{1,1}$ case we assume the coordinates in the definition --- in light of the following theorem the coordinate assumption is not superfluous.\NTR{Added this sentence to clarify the role of global coordinates.}

\begin{theorem}
\label{thm:c11-simplicity}
In smooth geometry definitions~\ref{def:simple-a} and~\ref{def:simple-b} are equivalent in the following sense:
\begin{enumerate}
\item
If~$M$ is a simple~$C^\infty$ Riemannian manifold (see definition~\ref{def:simple-a}), then it is diffeomorphic to a closed ball in $\R^n$ and it is a simple~$C^{1,1}$ Riemannian manifold (see definition~\ref{def:simple-b}).
\item
If~$M$ is a simple~$C^{1,1}$ Riemannian manifold (see definition~\ref{def:simple-b}) and its metric tensor is~$C^\infty$-smooth, then~$M$ is a smooth simple Riemannian manifold (see definition~\ref{def:simple-a}).
\end{enumerate}
\end{theorem}

\begin{remark}
\label{rmk:h=0}
The assumption $h|_{\partial M}=0$ in claim~\ref{injectivity-claim-2} of theorem~\ref{thm:c11-injectivity} is probably not necessary.
Not assuming this is fine in smooth geometry but leads to technical difficulties in our rough setup.
This added assumption is the only way in which our results fail to correspond to the classical smooth results.
\end{remark}

\subsection{Related results}

Geodesic X-ray transforms have been studied a lot on smooth manifolds equipped with~$C^{\infty}$-smooth Riemannian metrics.
Injectivity of the transform is reasonably well understood both on manifolds with a boundary and on closed manifolds.
On manifolds with boundary one integrates over maximal geodesics between two boundary points, whereas on closed manifolds one integrates over periodic geodesics.

After Mukhometov's introduction of the Pestov identity for scalar tomography~\cite{MukhometovIKPSP,MukhometovRPTDRM,MukhometovOPRRM}, the method has been applied to~$1$-forms and higher order tensor fields~\cite{ARUDFFDIAG,PSIGOTFMNC,PSUTTS,PSUIDBTTT} on many simple manifolds.
When one passes from simple manifolds with boundary to closed Anosov manifolds, the Pestov identity remains the same but the other tools around it change somewhat~\cite{CSSRCNCM,DSSPIGAM,PSUSRIDAS,PSUIDBTTT,UhlmannDBRSRRSNFP}.
Cartan--Hadamard manifolds are a non-compact analogue of simple manifolds, and the familiar Pestov identity works well~\cite{LehtonenGRTTDCHM,LRSTTCHM}.
Other variations of the problem change the Pestov identity, but a variant remains true and useful:
In the presence of reflecting rays a boundary term on the reflector is added~\cite{ISBRTRSCO,IPBRTORO},
an attenuation or a Higgs field~\cite{SUARTSS,PSUARTCHF,GPSXRTCNC} and magnetic flows~\cite{DPSUBRPPMF,AinsworthAMRTS,MPIPGD} add a term to the geodesic vector field,
non-abelian\NTR{Added hyphen.} versions of the problem remove the concept of a line integral entirely~\cite{FUXRTNACTD,PSNAXRTS,MNP:non-abelian}\NTR{Added one reference.},
and on Finsler surfaces a number of new terms are needed to account for non-Riemannian geometry~\cite{ADXRTGFCFS}.
On pseudo-Riemannian manifolds a Pestov identity useful for the light ray transform only seems to exist in product geometry of at least~$2+2$ dimensions~\cite{IlmavirtaXRTPR}.

Pestov identities are not the only tool in the box for studying ray transforms on manifolds.
For the variety of other methods we refer the reader to the review~\cite{IMIGMBA}.

Inverse problems in integral geometry have been mostly studied on manifolds whose Riemannian metric is smooth or otherwise substantially above our~$C^{1,1}$ in regularity.
Injectivity of the scalar X-ray transform is known on spherically symmetric manifolds of regularity~$C^{1,1}$ satisfying the so-called Herglotz condition when the conformal factor of the metric is in~$C^{1,1}$~\cite{HIKSRSSMB}.

Some geometric inverse problems outside integral geometry have been solved in low regularity.
A manifold with a metric tensor in a suitable Zygmund class is determined by its boundary spectral data~\cite{AKKLTBRREGCGIBP},
interior spectral data~\cite{BKLRSGIISP} or by its boundary distance function~\cite{KKLSBDRRRM}.

\subsection{Preliminaries}
\label{subsec:preliminaries}

In this subsection we will set up enough language to be able to state our definitions and give our proofs on a higher level.
For a similar framework in the traditional smooth setting, see e.g.~\cite{PSUIDBTTT}.
We will cover the foundations in more detail in section~\ref{sec:preliminaries} before embarking on the detailed proofs of our key lemmas.

The Riemannian manifold~$(M,g)$, where~$g$ is~$C^{1,1}$ regular, comes equipped with the unit sphere bundle~$\pi \colon SM \to M$.
The geodesic flow is a dynamical system on~$SM$ and its generator~$X$ is called the geodesic vector field.
Properties and coordinate representations of~$X$ will be given later.

We will make frequent use of the bundle~$N$ over~$SM$ defined next.
If~$\pi^{\ast}TM$ is the pullback of~$TM$ over~$SM$, then~$N$ is the subbundle of~$\pi^{\ast}TM$ with fibers $N_{(x,v)} = \{v\}^{\perp} \subseteq T_xM$. It is well known (see~\cite{PaternainGF}) that the tangent bundle~$TSM$ of~$SM$ has an orthogonal splitting
\begin{equation}
\label{eq:hv-split}
TSM = \R X \oplus \H \oplus \V
\end{equation}
with respect to the so-called Sasaki metric, where~$\H$ and~$\V$ are called horizontal and vertical subbundles respectively.
Roughly speaking,~$\H_{(x,v)}$ corresponds to derivatives on~$SM$ in the base without components in the direction of~$v$ and~$\V_{(x,v)}$ corresponds to derivatives on a fiber~$S_xM$.
It is natural to identify~$\H_{(x,v)} = N_{(x,v)}$ and~$\V_{(x,v)} = N_{(x,v)}$.

Given~$z \in SM$, let~$\gamma_z$ be the unique geodesic corresponding to the initial condition~$z$.
We define the geodesic flow to be the collection of (partially defined) maps~$\phi_t \colon SM \to SM$,~$\phi_t(z) = (\gamma_z(t),\dot\gamma_z(t))$, where~$t$ goes through the values for which the right side is defined on~$SM$.
For any~$z \in SM$ the geodesic~$\gamma_z$ is defined on a maximal interval $[\tau_-(z),\tau_+(z)]$.
The travel time function~$\tau \colon SM \to \R$ describes the first time a geodesic exists the manifold and it is defined by~$\tau(z) = \tau_+(z)$ for~$z \in SM$.
Clearly $\gamma_z(\tau(z)) \in \partial M$ for any~$z \in SM$.

A function~$f$ on~$M$ can be\NTR{Fixed typo.} identified with the function~$\pi^{\ast}f$ on~$SM$.
If~$h$ is a $1$-form on~$M$, then it can be considered as a function~$\tilde h \colon SM \to \R$ through the identification~$\tilde h(x,v) = h_x(v)$ for~$(x,v) \in SM$.
Since~$h_x \colon T_xM \to \R$ is linear,~$\tilde h$ uniquely corresponds to~$h$.
The integral function~$u^f \colon SM \to \R$ of $f\in\Lip(SM)$ is defined by
\begin{equation}
\label{eq:uf}
u^f(x,v)
\coloneqq
\int_0^{\tau(x,v)}
f(\phi_t(x,v))
\,\d t
\end{equation}
for all $(x,v)\in SM$.

The lift of a unit speed curve $\gamma\colon I\to M$ is $\tilde\gamma\colon I\to SM$ given by $\tilde\gamma(t)=(\gamma(t),\dot\gamma(t))$.
The curve~$\gamma$ is a geodesic if and only if the lift satisfies $\dot{\tilde\gamma}(t) = X(\tilde\gamma(t))$.
The geodesic vector field~$X$ acts naturally on scalar fields by differentiation, and on sections~$V$ of~$N$ it acts by\NTR{Fixed typo.}
\begin{equation}
XV(z) = D_tV(\phi_t(z))|_{t = 0},
\end{equation}
where~$D_t$ is the covariant derivative along the curve~$t \mapsto \gamma_z(t)$.
This operator maps indeed sections of~$N$ to sections of~$N$.

According to~\eqref{eq:hv-split} the gradient of a~$C^1$ function~$u$ on~$SM$ we can be written as
\begin{equation}
\nabla_{SM}u = (Xu)X + \grad{h}u + \grad{v}u.
\end{equation}
This gives rise to two new differential operators~$\grad{v}$ and~$\grad{h}$, called, respectively, the vertical and the horizontal gradient.
Both~$\grad{v}u$ and~$\grad{h}u$ are naturally interpreted as sections of~$N$; see~\cite{PSUGIPETD} for details.
There are natural~$L^2$ spaces for functions on the sphere bundle as well as for the sections of the bundle~$N$.
These will be denoted~$L^2(SM)$ and~$L^2(N)$ and we will often label the corresponding inner products explicitly.
Formal adjoints of~$\grad{v}$ and~$\grad{h}$ with respect to appropriate~$L^2$ inner products are the vertical and horizontal divergences~$-\dive{v}$ and~$-\dive{h}$ respectively.
The mapping properties of the operators in~$C^{1,1}$ regular metric setting are
\begin{align}
X &\colon C^1(SM) \to C(SM)
\\
X &\colon C^1(N) \to C(N),
\\
\grad{v},\grad{h} &\colon C^1(SM) \to C(N), \quad\text{and}
\\
\dive{v}, \dive{h} &\colon C^1(N) \to C(SM).
\end{align}
These mapping properties are easily verified by inspecting the explicit formulas in local coordinates; see section~\ref{sec:preliminaries}.

We will deal with Sobolev spaces~$H^1_{(0)}(SM)$ and~$H^1_{(0)}(N)$ defined as completions of~$C^1_{(0)}$ regular functions or sections in the relevant norms (see section~\ref{sec:preliminaries}), where the optional subscript~$0$ indicates zero boundary values.
Similarly, we denote by~$\Lip_0(M)$ and~$\Lip_0(SM)$ the spaces of Lipschitz functions zero boundary values.\NTR{Added this definition of $\Lip_0$.}
As the last function space we introduce a Sobolev space~$H^1_{(0)}(N,X)$, which only gives control over the operator~$X$ operating on sections of~$N$. From definitions of various Sobolev norms it will be clear that all differential operators are bounded~$H^1 \to L^2$ and thus extend to operators between Sobolev spaces.

Finally, there is a special quadratic form~$Q$ appearing in the Pestov identity.
To define it, we use the Riemannian curvature tensor
$R
\colon
L^{\infty}(N)
\to
L^{\infty}(N)
$
acting on sections of~$N$ by
\begin{equation}
R(x,v)V(x,v)
=
R(V(x,v),v)v.
\end{equation}
In order to verify the mapping property of~$R$, observe that the second partial derivatives of $g\in C^{1,1}=W^{2,\infty}$ are in~$L^{\infty}$.
We define~$Q$ by letting
\begin{equation}
Q(W)
=
\norm{XW}^2_{L^2(N)}
-
\iip{RW}{W}_{L^2(N)}.
\end{equation}
for all~$W \in H^1(N,X)$.

To conclude the preliminaries we recall in definition~\ref{def:simple-a} the traditional definition of a simple Riemannian manifold (cf.~\cite{PSUIDBTTT}).
In what follows a manifold satisfying conditions~\ref{a1} and~\ref{a2} is called \emph{simple~$C^\infty$ manifold}.
In definition~\ref{def:simple-b} we redefine the notion of simplicity on manifolds equipped with non-smooth Riemannian metrics.

\begin{definition}[Simple~$C^\infty$ manifold]
\label{def:simple-a}
Let~$(M,g)$ be a compact smooth Riemannian manifold with a smooth boundary. The manifold~$(M,g)$ is called \emph{simple~$C^\infty$ Riemannian manifold}, if the following hold:
\begin{enumerate}[label=A{{\arabic*}}:, ref=A{{\arabic*}}]
    \item\label{a1} The boundary~$\partial M$ is strictly convex in the sense of the second fundamental form.      
    \item\label{a2} Any two points on~$M$ can be joined by a unique geodesic in the interior of~$M$, and its length depends smoothly on its end points.
\end{enumerate}
\end{definition}

\begin{definition}[Simple~$C^{1,1}$ manifold]
\label{def:simple-b}
Let~$M\subseteq\R^n$ be the closed unit ball and~$g$ a~$C^{1,1}$ regular Riemannian metric on~$M$.
We say that~$(M,g)$ is a \emph{simple~$C^{1,1}$ Riemannian manifold} if the following hold:
\begin{enumerate}[label=B{{\arabic*}}:, ref=B{{\arabic*}}]
    \item\label{b1} There is~$\varepsilon > 0$ so that ~$Q(W) \ge \varepsilon\norm{W}^2_{L^2(N)}$ for all~$W \in H_0^1(N,X)$.
    \item\label{b2} Any two points of~$M$ can be joined by a unique geodesic in the interior of~$M$, whose length depends continuously on its end points.
    \item\label{b3} The function~$\tau^2$ is Lipschitz on~$SM$.
\end{enumerate}
\end{definition}

\begin{remark}
\label{rem:global-coords}
In definition~\ref{def:simple-b} the assumption that~$M$ is the closed unit ball is not restrictive --- any simple $C^\infty$ Riemannian manifold is diffeomorphic to a closed ball in a Euclidean space.
In the absence of conjugate points the exponential map~$\exp_x$, related to an interior point $x \in\sisus(M)$, maps its maximal domain~$D_x$ diffeomorphically to~$M$ and~$D_x$ is itself diffeomorphic to the closed unit ball in~$\R^n$ (see~\cite{PSUGIPETD}).
We use global coordinates on a simple~$C^{1,1}$ Riemannian manifold and we have decided to include their existence in the definition.
\end{remark}

\begin{remark}
If one is to define a rough simple manifold as the limit of smooth simple manifolds, the simplicity needs to be quantified.
The example of a hemisphere as the limit of expanding polar caps shows that the smooth limit of smooth simple manifolds can be a smooth but non-simple manifold.\NTR{Recast this sentence for clarity.}
The limit procedure can introduce conjugate points and failure of strict convexity on the boundary.
An example of quantified simplicity can be found in~\cite{HILSSRSRMUIS}, but we do not take this limit route in our definition here.
\end{remark}

\subsection{Acknowledgements}

Both authors were supported by the Academy of Finland (JI by grants 332890 and 351665, AK by 336254).
We thank Matti Lassas for discussions and the anonymous referees for useful remarks\NTR{Added this. Many thanks!}.

\section{Proof of theorem~\ref{thm:c11-injectivity}}

This section contains the proof of theorem~\ref{thm:c11-injectivity}.
The proofs of the necessary lemmas are postponed to section~\ref{sec:lemmas-in-low-reg}.
More detailed definitions of function spaces and operators can be found from section~\ref{sec:preliminaries}.

We will freely identify a scalar function~$f$ and a one-form~$h$ on~$M$ with scalar functions on~$SM$ as described above.
Interpreting~$f$ and~$h$ as functions on $SM$ we can apply formula~\eqref{eq:uf} to both.

\begin{lemma}[Regularity of integral functions]
\label{lemma:regularity-of-integral-functions}
Let~$(M,g)$ be a simple~$C^{1,1}$ manifold.
\begin{enumerate}[(1)]
\setlength\itemsep{0.2em}
\item\label{reg-of-uf} Let~$f$ be a Lipschitz function on~$M$ that integrates to zero over all maximal geodesics of~$M$ and let~$u^f$ be the integral function of~$f$ defined by~\eqref{eq:uf}. Then~$u^f \in \Lip_0(SM)$,~$Xu^f \in H^1(SM)$ and~$\grad{v}u^f \in H^1_0(N,X)$.
\item\label{reg-of-uh} Let~$h$ be a Lipschitz~$1$-form on~$M$ that integrates to zero over all maximal geodesics of~$M$ and vanishes on the boundary~$\partial M$. If~$u^h$ is the integral function of~$h$ defined by~\eqref{eq:uf}, then~$u^h \in \Lip_0(SM)$,~$Xu^h \in H^1(SM)$ and~$\grad{v}u^h \in H^1_0(N,X)$.
\end{enumerate}
\end{lemma}

\begin{lemma}[Pestov identity]
\label{lma:c11-pestov}
Let~$(M,g)$ be a simple~$C^{1,1}$ manifold and 
let~$u \in \Lip_0(SM)$
be such that~$Xu \in H^1(SM)$ and~$\grad{v}u \in H^1(N,X)$.
Then
\begin{equation}
\label{eq:pestov}
\norm{\grad{v}Xu}^2_{L^2(N)}
=
Q\left(\grad{v}u\right)
+
(n-1)
\norm{Xu}^2_{L^2(SM)}.
\end{equation}
\end{lemma}

Lemma~\ref{lemma:regularity-of-integral-functions} provides enough regularity to apply the Pestov identity~\eqref{eq:pestov} to the integral functions~$u^f$ and~$u^h$ because we will see in remark~\ref{rem:regularity-preserved} that~$\Lip(SM) \subseteq H^1(SM)$ even if the metric tensor in only in~$C^{1,1}$.
The following lemma shows that certain norms of the integral function~$u^h$ of a~$1$-form cancel in the identity.

\begin{lemma}
\label{lma:oneform-in-pestov-cancels}
Let~$(M,g)$ be a simple~$C^{1,1}$ manifold and let~$h$ be a Lipschitz~$1$-form on~$M$.
Then
\begin{equation}
\norm{\grad{v}h}^2_{L^2(N)}
=
(n-1)
\norm{h}^2_{L^2(SM)}.
\end{equation}
\end{lemma}

We are ready to prove theorem~\ref{thm:c11-injectivity}.

\begin{proof}[Proof of theorem~\ref{thm:c11-injectivity}]
\ref{injectivity-claim-1}
The integral function~$u^f$ of $f\in\Lip(M)$ satisfies
$Xu^f \in H^1(SM)$ and $\grad{v}u^f \in H^1(N,X)$ by lemma~\ref{lemma:regularity-of-integral-functions}.
Thus we can apply the Pestov identity of lemma~\ref{lma:c11-pestov} to~$u^f$.
By the fundamental theorem of calculus~$Xu^f = -f$ and thus~$\grad{v}Xu^f = 0$, since~$f$ does not depend on the direction~$v \in S_xM$.
By~$C^{1,1}$ simplicity (definition~\ref{def:simple-b}) of~$(M,g)$, the quadratic form~$Q$ is non-negative. Thus the Pestov identity reduces to
\begin{equation}
0
\ge
(n-1)
\norm{Xu^f}^2_{L^2(SM)}.
\end{equation}
Hence~$f = -Xu^f = 0$ in~$L^2(SM)$ as claimed.

\ref{injectivity-claim-2}
If~$h = \d p$ for some scalar function~$p \in C^{1,1}(M)$ with~$p|_{\partial M} = 0$, then by the fundamental theorem of calculus~$h$ integrates to zero over all maximal geodesics of~$M$.

Let~$h$ be a Lipschitz~$1$-form on~$M$ that integrates to zero over all maximal geodesic of~$M$ and vanishes on the boundary~$\partial M$.
We will show that~$h = \d p$ for some function~$p \in C^{1,1}(M)$ vanishing on~$\partial M$.
Lemma~\ref{lemma:regularity-of-integral-functions} allows us to apply the Pestov identity to the integral function~$u^h$ of~$h$.
Due to lemma~\ref{lma:oneform-in-pestov-cancels}, the identity reduces to
\begin{equation}
Q\left(\grad{v}u^h\right)
=
0
.
\end{equation}
Since the manifold is simple~$C^{1,1}$, this can only happen if~$\grad{v}u^h = 0$.
The function~$u^h$ is Lipschitz and independent of the direction~$v \in S_xM$ on each fiber and therefore there is a Lipschitz scalar function~$p$ on~$M$ so that~$u^h=-\pi^*p$ on~$SM$.
Additionally,~$p|_{\partial M} = u^h|_{\partial(SM)} = 0$, since~$h$ integrates to zero over all maximal geodesics of~$M$.
Since $Xu^h = -h$, we have shown that~$\d p = h$ in the weak sense.
Because~$h$ is Lipschitz-continuous by assumption, we have that~$\der p$ is Lipschitz and thus~$p\in C^{1,1}$ and the proof is complete.
\end{proof}

\section{Proof of theorem~\ref{thm:c11-simplicity}}
\label{sec:proof-of-second-thm}

In this section we prove that in the smooth setting definition~\ref{def:simple-b} of~$C^{1,1}$ simplicity is equivalent to definition~\ref{def:simple-a} of~$C^\infty$ simplicity.
Proofs of lemmas~\ref{lma:posit-q-implies-no-conj} and~\ref{lma:conv-bnd-equiv-tau2-lip} are given in section~\ref{sec:lemmas-in-smooth-geometry}.
Theorem~\ref{thm:c11-simplicity} readily follows from lemmas~\ref{lma:posit-q-implies-no-conj} and~\ref{lma:conv-bnd-equiv-tau2-lip}.

\begin{lemma}
\label{lma:posit-q-implies-no-conj}
Let~$(M,g)$ be a simple~$C^{1,1}$ manifold with~$C^{\infty}$-smooth Riemannian metric~$g$. Then there are no conjugate points in~$M$, not even on the boundary.
\end{lemma}

\begin{lemma}
\label{lma:conv-bnd-equiv-tau2-lip}
Let~$M$ be a compact Riemannian manifold with smooth boundary and a~$C^\infty$-smooth Riemannian metric~$g$.
Suppose that~$(M,g)$ is non-trapping. Then~$\partial M$ is strictly convex in the sense of the second fundamental form if and only if~$\tau^2 \in \Lip(SM)$.
\end{lemma}

\begin{proof}[Proof of theorem~\ref{thm:c11-simplicity}]
By remark~\ref{rem:global-coords} each simple~$C^{\infty}$ Riemannian manifold is diffeomorphic to the closed unit ball~$\overline{B}$ in~$\R^n$.
Thus we may assume that~$M = \overline{B}$ and let~$g$ be a~$C^\infty$-smooth Riemannian metric on~$M$.
It suffices to show that~$(M,g)$ satisfies conditions~\ref{a1}--\ref{a2} in definition~\ref{def:simple-a} if and only if it satisfies conditions \ref{b1}--\ref{b3} in definition~\ref{def:simple-b}.
We have illustrated these implications in figure~\ref{fig:proof-of-thm-2}.

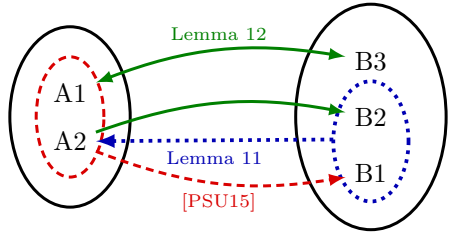
\begin{figure}[htb]
\centering
\begin{tikzpicture}
\node[] at (-4,0.325) (a1) {\ref{a1}};
\node[] at (-4,-0.325) (a2) {\ref{a2}};
\node[] at (0,-0.75) (b1) {\ref{b1}};
\node[] at (0,0) (b2) {\ref{b2}};
\node[] at (0,0.75) (b3) {\ref{b3}};
%

\draw[black,line width=0.4mm] (-4,0) ellipse (0.8cm and 1.2cm);
\draw[black,line width=0.4mm] (0,0) ellipse (1cm and 1.5cm);

\draw[blue!70!black,line width=0.5mm, dotted] (0,-0.325) ellipse (0.5cm and 0.8cm);
\draw[red!85!black,line width=0.4mm,densely dashed] (-4,0) ellipse (0.45cm and 0.8cm);

\draw[-{latex[round]},line width=0.4mm,red!85!black,densely dashed]
(a2) edge[bend right=15]node[midway,below]{\scriptsize \cite{PSUIDBTTT}} (b1);

\path[-{latex[round]},line width=0.5mm,blue!70!black, dotted]
(b2) + (-0.5,-0.3) edge node[midway,below]{\scriptsize Lemma \ref{lma:posit-q-implies-no-conj}} (a2);

\draw[-{latex[round]},line width=0.4mm,green!50!black] (a2) edge[bend left=15] (b2);
%

\draw[{latex[round]}-{latex[round]},line width=0.4mm,green!50!black]
(a1) edge[bend left=15] node[midway,above]{\scriptsize Lemma~\ref{lma:conv-bnd-equiv-tau2-lip}} (b3);
\end{tikzpicture}
\captionsetup{width=0.9\textwidth}
\caption{Illustration of the proof of theorem~\ref{thm:c11-simplicity}.
The arrows represent implications except the one double headed arrow, which represents equivalence.
The \textcolor{green!50!black}{green (solid)} arrows connect one condition to another.
The \textcolor{red!85!black}{red (dashed)} and the \textcolor{blue!70!black}{blue (dotted)} arrows indicate that one condition follows from the two conditions circled with the same color (style).}
\label{fig:proof-of-thm-2}
\end{figure}

By lemma~\ref{lma:conv-bnd-equiv-tau2-lip} conditions~\ref{a1} and~\ref{b3} are equivalent.
By lemma~\ref{lma:posit-q-implies-no-conj} the condition~\ref{b1} implies that there are no conjugate points on~$M$.
Thus we can promote the continuous dependence in~\ref{b2} to smooth dependence~\ref{a2}.
Therefore simple~$C^{1,1}$ manifolds satisfy both conditions~\ref{a1} and~\ref{a2} of~$C^{\infty}$ simplicity.
Conversely, simple~$C^\infty$ manifolds satisfy~\ref{b1} (see~\cite[Lemma 11.2]{PSUIDBTTT}) and clearly~\ref{b2} is strictly weaker than\NTR{Fixed typo.}~\ref{a2}.
\end{proof}

\section{Bundles, function spaces and operators}
\label{sec:preliminaries}

This section complements the preliminaries in subsection~\ref{subsec:preliminaries}.
The main focus is on a detailed description of structures, functions spaces and operators build on a compact Riemannian manifold~$(M,g)$ with a~$C^{1,1}$ regular Riemannian metric.

\subsection{Function spaces on smooth manifolds}

Let~$M$ be a compact smooth manifold with a smooth boundary. The space of smooth functions on~$M$ is denoted~$C^{\infty}(M)$ and the space of differentiable functions with Lipschitz derivatives is denoted~$C^{1,1}(M)$. We let~$C^{1,1}(T^2M)$ denote the space of~$2$-tensor fields on~$M$, whose component functions are in~$C^{1,1}(M)$.

If~$h$ is a smooth Riemannian metric on~$M$, then~$L^2_h(M)$ and~$L^{\infty}_h(M)$ will respectively denote spaces of square integrable and essentially bounded functions on~$M$, where the Riemannian volume form of~$h$ is used as the measure.
Similarly,~$W^{1,p}_h(M)$ and~$W^{2,p}_h(M)$ will respectively denote Sobolev spaces with~$p$-integrable covariant derivatives of the first order and of the second order.
Norms of the covariant derivatives on the tangent spaces are always defined by the metric~$h$.

\subsection{Structures in low regularity}

Let~$(M,g)$ be a compact Riemannian manifold with a smooth boundary.
We assume that~$g \in C^{1,1}(T^2M)$.
The unit sphere bundle $SM=\{v\in TM:\abs{v}=1\}$ is a submanifold of~$TM$, but not in general a smooth one.
Despite the non-smoothness of~$SM\subseteq TM$ as a submanifold, it can be equipped with an induced smooth structure:~$SM$ is naturally homeomorphic to the quotient space~$(TM\setminus 0)/\sim$, where~$v \sim \lambda v$ for all~$\lambda > 0$ and~$v \in T_xM$.
Metric structures like the Sasaki metric are still non-smooth, so this smooth structure is of little use.
We will only see~$SM$ as a submanifold of~$TM$.

For $k \in \{0,1\}$ a function~$u \colon SM \to \R$ is said to be in~$C^{k}(SM)$ if~$u$ is~$k$ times continuously differentiable --- for $k\geq2$ this concept is undefined in our setting.
As a~$C^1$ submanifold of~$TM$ the sphere bundle has enough regularity to define both~$C(SM)$ and~$C^1(SM)$.
The subset~$C^k_0(SM)$ of~$C^k(SM)$ consists of functions vanishing on
\begin{equation}
\partial(SM) = \{\, (x,v) \in SM \,:\, x \in \partial M \,\}.
\end{equation}
The set of Lipschitz functions on~$SM$ is denoted by~$\Lip(SM)$.
We denote the inward unit normal vector field to the
boundary~$\partial M$ by~$\nu$. The boundary~$\partial(SM)$ is divided into parts pointing inwards and outwards, respectively denoted by
\begin{equation}
\dooin(SM)
\coloneqq
\{\, (x,v) \in \partial(SM) \,:\, \ip{v}{\nu(x)} \ge 0 \,\}
\end{equation}
and
\begin{equation}
\dooout(SM)
\coloneqq
\{\, (x,v) \in \partial(SM) \,:\, \ip{v}{\nu(x)} \le 0 \,\}.
\end{equation}
Their intersection consists of tangential directions
\begin{equation}
\partial_0(SM)
\coloneqq
\dooin(SM) \cap \dooout(SM).
\end{equation}

Many differential operators considered in this article operate on sections of the bundle~$N$. To describe~$C^k$ spaces of sections of~$N$, recall that~$N$ is the subbundle of~$\pi^{\ast}TM$ with fibers $N_{(x,v)} = \{v\}^{\perp} \subseteq T_xM$.
A section~$V$ of the bundle~$N$ is a section of the bundle~$\pi^{\ast}TM$ with the property that $\ip{V(x,v)}{v}_{g(x)} = 0$ for all~$(x,v) \in SM$.
We say that such a section is in~$C^k(N)$ for~$k \in \{0,1\}$ if the corresponding section of~$\pi^{\ast}TM$ is~$k$ times continuously differentiable.
Differentiability of a section~$W$ of~$\pi^{\ast}TM$ is well defined since~$W$ is a certain function between two differentiable manifolds~$SM$ and~$TM$.
The subspace~$C^k_0(N) \subseteq C^k(N)$ consists of sections~$V$ of~$N$ that vanish on~$\partial(SM)$.

Let~$(x,v)$ be a local coordinate system on~$TM$ and let~$\partial_{x^j}$ and~$\partial_{v^k}$ be corresponding coordinate vector fields.
We introduce new vector fields~$\delta_{x^j} = \partial_{x^j} - \Gamma^l_{\ jk}v^k\partial_{v^l}$ on~$TM$, where~$\Gamma^l_{\ jk}$ are the Christoffel symbols of the metric~$g$.
As the metric tensor in our results is of regularity~$C^{1,1}$, it follows that the Christoffel symbols and thus the vector fields~$\delta_{x^j}$ are only Lipschitz.\NTR{Added this sentence to point out the regularity.}


\subsection{Differential operators}
\label{subsec:differentialoperators}

Next we define differential operators on~$SM$ and~$N$.
The basic coordinate derivatives of a function~$u \in C^1(SM)$ are defined by
\begin{equation}
\delta_ju \coloneqq \delta_{x^j}(u \circ r)|_{SM}
\quad
\text{and}
\quad
\partial_ku \coloneqq \partial_{v^k}(u \circ r)|_{SM},
\end{equation}
where~$r \colon TM \setminus 0 \to SM$ is the radial function~$r(x,v) = (x,v\abs{v}^{-1}_{g(x)})$.
We denote~$\delta^j \coloneqq g^{jk}\delta_k$ and~$\partial^j \coloneqq g^{jk}\partial_k$.
We use the basic derivatives to define operators in local coordinates.

The geodesic vector field~$X$ is a differential operator that acts both on functions on~$SM$ and on sections of the bundle~$N$. The actions on a scalar function~$u$ and on a section~$V$ are defined by
\begin{equation}
\label{eq:geod-vector-field}
Xu = v^j\delta_ju
\quad
\text{and}
\quad
XV = (XV^j)\partial_{x^j} + \Gamma^l_{\ jk}v^jV^k\partial_{x^l}.
\end{equation}
Vertical and horizontal gradients are differential operators defined respectively by
\begin{equation}
\grad{v}u = (\partial^ju)\partial_{x^j}
\quad
\text{and}
\quad
\grad{h}u = (\delta^ju - (Xu)v^j)\partial_{x^j}.
\end{equation}
Coordinate formulas indicate that~$\grad{v}$ is the gradient in~$v$ and~$\grad{h}$ is the gradient in~$x$ with the direction of~$v$ being projected out. The adjoint operators of~$\grad{v}$ and~$\grad{h}$ are the vertical and the horizontal divergences
\begin{align}
\dive{v} V = \partial_jV^j
\quad
\text{and}
\quad
\dive{h} V = (\delta_j + \Gamma^i_{\ ji})V^j.
\end{align}

The Riemannian curvature tensor~$R$ of the metric~$g$ has an action on sections of~$N$ defined by
\begin{equation}
RV
=
R^l_{\ ijk}V^iv^jv^k\partial_{x^l}.
\end{equation}

\begin{figure}[htb]
\centering
\begin{tikzpicture}[every text node part/.style={align=center}]
\node[text width=2.5cm] at (-2.1,0) (a1) {Functions on~$SM$};
\node[text width=2cm] at (2.1,0) (a2) {Sections of~$N$};
\node[circle,draw=none,minimum size=1cm] at (2.1,-0.45) (b1) {};
\node[circle,draw=none,minimum size=1cm] at (2.1,0.45) (b2) {};
\node[circle,draw=none,minimum size=1cm] at (-2.1,0.45) (b3) {};

\draw[black,line width=0.3mm] (-2.1,0) ellipse (1cm and 0.7cm);
\draw[black,line width=0.3mm] (2.1,0) ellipse (1cm and 0.7cm);

\draw[-{latex[round]},line width=0.3mm]
(a1) edge[bend left=20] node[midway,above]{$\grad{v}, \grad{h}$} (a2);
\draw[-{latex[round]},line width=0.3mm]
(a2) edge[bend left=20] node[midway,below]{$\dive{v}$, $\dive{h}$} (a1);
%

\draw[-{latex[round]},line width=0.3mm,]
(b2) edge[loop above,style={min distance=1.5cm,in=55,out=125,looseness=8}] node[midway,above]{$X$} (b2);
\draw[-{latex[round]},line width=0.3mm]
(b3) edge[loop above,style={min distance=1.5cm,in=55,out=125,looseness=8}] node[midway,above]{$X$} (b3);
\draw[-{latex[round]},line width=0.3mm,]
(b1) edge[loop above,style={min distance=1.5cm,in=235,out=305,looseness=8}] node[midway,below]{$R$} (b1);
\end{tikzpicture}
\captionsetup{width=0.9\textwidth}
\caption{Interplay of the operators defined in subsection~\ref{subsec:differentialoperators}. The gradients map functions on~$SM$ to sections of~$N$. The divergences map sections of~$N$ back to function on~$SM$. The geodesic vector field maps functions to functions and sections to sections. The curvature operator acts only on sections and produces sections.}
\label{fig:figure2}
\end{figure}
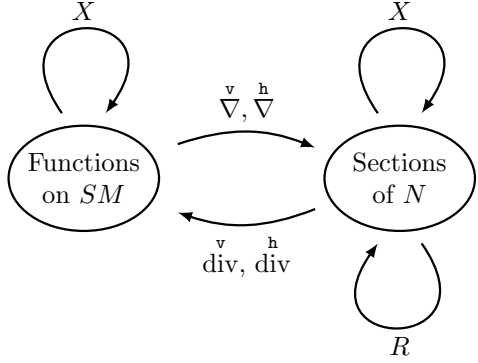

\subsection{Integration and Sobolev spaces}

A simple~$C^{1,1}$ manifold~$M$ is orientable, so the Riemannian volume form on it can be defined in local coordinates as
\begin{equation}
\d V_g(x) \coloneqq \abs{\det(g(x))}^{1/2}\d x^1 \wedge \dots \wedge \d x^n.
\end{equation}
For any~$x \in M$ the pair~$(S_xM,g(x))$ is a Riemannian manifold.
Let~$\d S_x$ be the associated Riemannian volume form on~$S_xM$.
We use~$\d V_g$ and~$\d S_x$ to define the volume form~$\d \Sigma_g$ on~$SM$,
given in local coordinates by
\begin{equation}
\d \Sigma_g(x,v)
=
\d S_x(v) \wedge \d V_g(x)
.
\end{equation}
The form~$\d \Sigma_g$ is natural as it coincides with the Riemannian volume form of the Sasaki metric on~$SM$.
Since~$\d V_g$ has as much regularity as~$g$, so does~$\d \Sigma_g$.

The~$L^2$-norm of a scalar function~$u$ on~$SM$ is denoted by~$\norm{u}_{L^2(SM)}$ and the~$L^2$-norm of a section~$V$ of~$N$ is denoted by~$\norm{V}_{L^2(N)}$. The~$L^2$-norms are induced by the inner products
\begin{equation}
\iip{u}{w}_{L^2(SM)}
\coloneqq
\int_{SM}
uw
\,\d \Sigma_g
\end{equation}
and
\begin{equation}
\iip{V}{W}_{L^2(N)}
\coloneqq
\int_{SM}
g_{ij}V^iW^j
\,\d \Sigma_g.
\end{equation}
We define the~$\norm{\cdot}_{H^1(SM)}$-norm of a function~$u \in C^1(SM)$ by
\begin{equation}
\norm{u}_{H^1(SM)}^2
=
\norm{u}_{L^2(SM)}^2
+
\norm{Xu}_{L^2(SM)}^2
+
\norm{\grad{v}u}_{L^2(N)}^2
+
\norm{\grad{h}u}_{L^2(N)}^2.
\end{equation}
The Sobolev space~$H^1(SM)$ is defined to be the completion of the subset of~$C^1(SM)$ that consists of functions with finite~$H^1(SM)$-norm. We denote by~$H^1_0(SM)$ the closure of~$C^1_0(SM)$ in~$H^1(SM)$.

Sobolev spaces for sections of~$N$ are defined in an analogous fashion. For a section~$V \in C^1(N)$ we define the two Sobolev norms
\begin{align}
\norm{V}_{H^1(N)}^2
=
\norm{V}_{L^2(N)}^2
+
\norm{XV}_{L^2(N)}^2
+
\norm{\dive{v}V}_{L^2(SM)}^2
+
\norm{\dive{h}V}_{L^2(SM)}^2
\end{align}
and
\begin{equation}
\norm{V}_{H^1(N,X)}^2
=
\norm{V}_{L^2(N)}^2
+
\norm{XV}_{L^2(N)}^2.
\end{equation}
The corresponding Sobolev spaces (the completions of~$C^1(N)$ under these norms) are denoted by~$H^1(N)$ and~$H^1(N,X)$, and the Sobolev spaces of sections vanishing on the boundary~$\partial(SM)$ are denoted by~$H^1_0(N)$ and~$H^1_0(N,X)$.

\begin{remark}
Contrary to what one might expect, the norm on~$H^1(N)$ defined above does not contain derivatives in all possible directions, as it only includes divergences in the vertical and horizontal directions.
We will use these norms only to estimate from above, so this omission of derivatives makes no difference.
\end{remark}

In the case where~$g$ is a~$C^{\infty}$-smooth Riemannian metric, we introduce one more Sobolev space,~$\SSob(SM)$. The defining norm on the dense subspace~$C^2(N)$ is
\begin{equation}
\norm{u}_{\SSob(SM)}^2
=
\norm{u}_{H^1(SM)}^2
+
\norm{Xu}_{H^1(SM)}^2
+
\norm{\grad{v}u}_{H^1(N,X)}^2.
\end{equation}

\begin{remark}
It is important to realize that we cannot define Sobolev spaces using smooth test functions as in the smooth case.
The reason is two-fold.
First, the natural structure of~$SM$ as an submanifold~$TM$ is not regular enough to define the function class~$C^{\infty}(SM)$.
Second, the differential operators themselves are not smooth.
Applying any of the differential operators immediately drops regularity to that of the coefficients, and they involve the metric tensor.
\end{remark}

\subsection{Differential operators on Sobolev spaces}

It is clear from the definitions that all of our differential operators are bounded~$H^1 \to L^2$. Thus all classically defined operators extend to operators between Sobolev spaces.
We therefore have the continuous operators
\begin{align}
X 
&\colon
H^1(SM) \to L^2(SM),
\\
X
&\colon
H^1(N) \to L^2(N),
\\
\grad{v},\grad{h}
&\colon
H^1(SM) \to L^2(N),
\quad\text{and}
\\
\dive{v},\dive{h}
&\colon
H^1(N) \to L^2(SM).
\end{align}

Basic integration by parts holds for the extended operators:
If~$u,w \in H^1(SM)$ and~$V,W \in H^1(N)$ and~$w$ and~$W$ vanish on the boundary, then
\begin{align}
\iip{Xu}{w}_{L^2(SM)}
&=
-\iip{u}{Xw}_{L^2(SM)},
\\
\iip{XV}{W}_{L^2(N)}
&=
-\iip{V}{XW}_{L^2(N)},
\\
\iip{\grad{v}u}{W}_{L^2(N)}
&=
-\iip{u}{\dive{v}W}_{L^2(SM)},
\quad\text{and}
\\
\iip{\grad{h}u}{W}_{L^2(N)}
&=
-\iip{u}{\dive{h}W}_{L^2(SM)}.
\end{align}
We can use the space~$C^1_0(SM)$ as test functions and~$C^1_0(N)$ as test sections.

\subsection{Switching between different unit sphere bundles}

Suppose we have two Riemannian metrics~$g,h \in C^{1,1}(T^2M)$ on the manifold~$M$.
Let~$S_gM$ and~$S_hM$ denote the corresponding unit sphere bundles.
There is a natural radial~$C^{1,1}$-diffeomorphism
\begin{equation}
s \colon S_gM \to S_hM, \quad s(x,v) = (x,v\abs{v}^{-1}_h).
\end{equation}
In section~\ref{sec:lemmas-in-low-reg} we will have three Riemannian metrics~$g \in C^{1,1}(T^2M)$ and~$\alf{g},h \in C^{\infty}(T^2M)$ with certain roles.
In this case we will denote the corresponding radial~$C^{1,1}$-diffeomorphisms by 
\begin{equation}
\alf{s} \colon S_{\alf{g}}M \to S_hM,
\quad
\alf{r} \colon S_gM \to S_{\alf{g}}M
\quad\text{and}\quad
s \colon S_gM \to S_hM.
\end{equation}
In section~\ref{sec:lemmas-in-low-reg} we frequently use the convention that the bundles related to~$\alf{g}$ are denoted~$\alf{S}M \coloneqq S_{\alf{g}}M$ and~$\alf{N} \coloneqq N_{\alf{g}}$, the operators related to~$\alf{g}$ are decorated with~$\alpha$ on top or as a subscript, the bundles and the operators related to~$h$ are decorated with subscripts~$h$, and the bundles and the operators related to the metric~$g$ are written without decorations.

\begin{remark}
\label{rem:regularity-preserved}
We can switch between sphere bundles and corresponding Sobolev spaces using pullbacks along radial functions.
If~$u$ is a scalar function on~$SM$, then~$(s^{-1})^*u$\NTR{Changed $s$ to $s^{-1}$.} is a scalar function on~$S_hM$.
To see that pullback behaves well on the Sobolev scale, note that the~$H^1(SM)$-norm controls all possible derivatives on~$SM$ since $TSM = \R X \oplus \V \oplus \H$.
Thus the~$H^1(SM)$-norm is equivalent to
\begin{equation}
\label{eq:alt-norm}
\norm{u}
=
\norm{u}_{L^2(SM)}
+
\norm{\d_{SM} u}_{L^2(T^*SM)}
\end{equation}
with the norm of the differential interpreted with respect to any Riemannian metric on~$SM$.
With the norm~\eqref{eq:alt-norm} we see that regularity on Sobolev scale is preserved, since $(s^{-1})^{\ast}(\d_{SM}u) = \d_{S_hM}(u \circ s^{-1})$\NTR{Changed $s$ to $s^{-1}$.} by standard properties of the pullback.\NTR{Recast sentence.}
\end{remark}

Remark~\ref{rem:regularity-preserved} allows us to prove continuous Sobolev embeddings between Sobolev spaces of low regularity metrics.
We present one example that will be useful to us later.
Let~$g \in C^{1,1}(T^2M)$ and~$h \in C^\infty(T^2M)$ be two Riemannian metrics on~$M$.
If~$u \in \Lip(SM)$, then~$(s^{-1})^{\ast}u \in \Lip(S_hM)$.
\NTR{Switched $s^{\ast}u$ to $(s^{-1})^{\ast}u$ so that it maps the right way.}
Since~$h$ is~$C^{\infty}$-smooth, we have~$(s^{-1})^{\ast}u \in H^1(S_hM)$\NTR{As above.}.
Then since~$\norm{u}_{H^1(SM)} \lesssim \norm{(s^{-1})^{\ast}u}_{H^1(S_hM)}$\NTR{As above.} by remark~\ref{rem:regularity-preserved}, we see that~$u \in H^1(SM)$.
We have shown that the inclusion~$\Lip(SM) \subseteq H^1(SM)$ holds even when the metric tensor is only~$C^{1,1}$.

\section{Lemmas in low regularity}
\label{sec:lemmas-in-low-reg}

\subsection{The Pestov identity}

In this subsection~$(M,g)$ is a simple~$C^{1,1}$ Riemannian manifold.
We prove a variant of the commutator formula $[X,\grad{v}] = -\grad{h}$ and the Pestov identity on~$(M,g)$.
First, we show that both results are valid for Sobolev functions on a manifold equipped with a~$C^{\infty}$-smooth Riemannian metric.
Then we show that only~$C^{1,1}$ regularity of the Riemannian metric is needed. The main focus of the subsection is on proving the Pestov identity of lemma~\ref{lma:c11-pestov}.



\begin{lemma}
\label{lma:cinfty-commutator}
Let~$(M,h)$ be a compact smooth manifold with a smooth boundary, where~$h$ is a~$C^{\infty}$-smooth Riemannian metric.
The commutator formula~$[X,\grad{v}] = -\grad{h}$ holds in the~$H^1$ sense on~$(M,h)$:
For~$u \in H^1_0(S_hM)$ and~$V \in C^1(N_h)$, we have
\begin{equation}
\iip{\grad{h}_hu}{V}_{L^2(N_h)}
=
\iip{\grad{v}_hu}{X_hV}_{L^2(N_h)}
-
\iip{X_hu}{\dive{v}_hV}_{L^2(S_hM)}.
\end{equation}
\end{lemma}

\begin{proof}
Let~$u \in H^1_0(S_hM)$ and~$V \in C^{\infty}(N_h)$.
Since~$V$ is smooth, by~\cite[Lemma 2.1.]{PSUIDBTTT} we have
\begin{equation}
X_h\dive{v}_hV
-
\dive{v}_hX_hV
=
-\dive{h}_hV.
\end{equation}
Thus
\begin{equation}
\begin{split}
\iip{\grad{h}_hu}{V}_{L^2(N_h)}
&=
-
\iip{u}{\dive{v}_hX_hV}_{L^2(S_hM)}
+
\iip{u}{X\dive{v}_hV}_{L^2(S_hM)}
\\
&=
\iip{\grad{v}_hu}{X_hV}_{L^2(N_h)}
-
\iip{X_hu}{\dive{v}_hV}_{L^2(S_hM)},
\end{split}
\end{equation}
where the last equality holds since~$u \in H^1_0(S_hM)$, and since~$X_hV \in C^{\infty}(N_h)$ and~$\dive{v}_hV \in C^{\infty}(S_hM)$.
The same identity holds for~$V \in C^1(N_h)$ by approximation, since only first order derivatives appear in the statement.
\end{proof}

\begin{lemma}
\label{lma:u-sobolev-pestov}
Let~$(M,h)$ be a compact smooth manifold with a smooth boundary, where~$h$ is a~$C^{\infty}$-smooth Riemannian metric.
Suppose that~$u \in K^2(S_hM)$ vanishes on the boundary~$\partial(S_hM)$.
Then
\begin{equation}
\label{eq:u-sobolev-pestov}
\norm{\grad{v}_hX_hu}^2_{L^2(N_h)}
=
Q_h\left(\grad{v}_hu\right)
+
(n-1)
\norm{X_hu}^2_{L^2(S_hM)},
\end{equation}
where~$Q_h$ is the quadratic form defined for~$W \in H^1_0(N_h,X_h)$ by
\begin{equation}
Q_h(W)
=
\norm{X_hW}^2_{L^2(N_h)}
-
\iip{R_hW}{W}_{L^2(N_h)}.
\end{equation}
\end{lemma}

\begin{proof}
Since~$u \in \SSob(S_hM)$ and~$u$ vanishes on the boundary~$\partial(S_hM)$, there is a sequence~$\big(\bet{u}\big)_{\beta \in \N}$ of smooth functions on~$S_hM$ vanishing on~$\partial(S_hM)$ so that~$\bet{u} \to u$ in~$\SSob(S_hM)$.
We see that
\begin{equation}
\label{eq:proof-of-u-sob-p-1}
\norm{\grad{v}_hX_h\bet{u} - \grad{v}_hX_hu}^2_{L^2(N_h)}
\le
\norm{X_h\bet{u} - X_hu}^2_{H^1(S_hM)}
\le
\norm{\bet{u} - u}^2_{\SSob(S_hM)}
\end{equation}
and
\begin{equation}
\label{eq:proof-of-u-sob-p-2}
\norm{X_h\bet{u} - X_hu}^2_{L^2(N_h)}
\le
\norm{\bet{u} - u}^2_{H^1(S_hM)}
\le
\norm{\bet{u} - u}^2_{\SSob(S_hM)}.
\end{equation}
Therefore~$\grad{v}_hX_h\bet{u} \to \grad{v}_hX_hu$ in~$L^2(N_h)$ and~$X_h\bet{u} \to X_hu$ in~$L^2(S_hM)$ as $\beta \to \infty$. Additionally, since the curvature operator~$R$ of the metric~$h$ continuously maps $L^{\infty}(N_h) \to L^{\infty}(N_h)$, we have
\begin{equation}
\label{eq:proof-of-u-sob-p-3}
\norm{\grad{v}_h\bet{u} - \grad{v}_hu}^2_{L^2(N_h)}
\le
\norm{\bet{u} - u}^2_{H^1(S_hM)}
\le
\norm{\bet{u} - u}^2_{\SSob(S_hM)}
\end{equation}
and
\begin{equation}
\label{eq:proof-of-u-sob-p-4}
\norm{R_h\grad{v}_h\bet{u} - R_h\grad{v}_hu}^2_{L^2(N_h)}
\lesssim
\norm{\bet{u} - u}^2_{H^1(S_hM)}
\le
\norm{\bet{u} - u}^2_{\SSob(S_hM)}.
\end{equation}
Thus~$Q_h\Big(\grad{v}_h\bet{u}\Big) \to Q_h\Big(\grad{v}_hu\Big)$ as~$\beta \to \infty$.
By the Pestov identity for smooth functions and metrics (see~\cite[Remark 2.3.]{PSUIDBTTT}) we have
\begin{equation}
\label{eq:u-sobolev-pestov-proof}
\norm{\grad{v}_hX_h\bet{u}}^2_{L^2(N_h)}
=
Q_h\left(\grad{v}_h\bet{u}\right)
+
(n-1)
\norm{X_h\bet{u}}^2_{L^2(S_hM)}.
\end{equation}
We now let $\beta \to \infty$ in~\eqref{eq:u-sobolev-pestov-proof}.
By our estimates~\eqref{eq:proof-of-u-sob-p-1},~\eqref{eq:proof-of-u-sob-p-2},~\eqref{eq:proof-of-u-sob-p-3}, and~\eqref{eq:proof-of-u-sob-p-4} we end up with the claimed identity~\eqref{eq:u-sobolev-pestov}.
\end{proof}

The rest of this section focuses on showing that we can replace the~$C^{\infty}$-smooth Riemannian metric~$h$ in lemmas~\ref{lma:cinfty-commutator} and~\ref{lma:u-sobolev-pestov} by a~$C^{1,1}$ regular Riemannian metric.
Let~$(M,g)$ be a~$C^{1,1}$ simple Riemannian manifold.
Next, we construct approximations of~$g$ by~$C^{\infty}$-smooth Riemannian metrics~$\alf{g}$.

Let $(x^1,\dots,x^n)$ be the usual Cartesian coordinates on the Euclidean closed ball $M\subset\R^n$ and extend all components~$g_{ij} \in C^{1,1}(M)$ of~$g$ to functions~$\overline{g}_{ij} \in C^{1,1}(\R^n)$. Such extensions exist since~$C^{1,1} = W^{2,\infty}$ and the boundary of~$M$ is smooth (see~\cite[Chapter~6,~Theorem~5]{SteinSIDPF}).
Let us then choose a non-negative compactly supported smooth function~$\varphi \colon \R^n \to \R$ with unit integral and define a sequence of standard mollifiers~$\alf{\varphi}(x) = \alpha^n\varphi(\alpha x)$ for~$\alpha \in \N$.
Then we define
\begin{equation}
\label{eq:sequence-of-approximating-metrics}
\alf{g}_{ij}
\coloneqq
(\alf{\varphi}\,\ast\,\overline{g}_{ij})|_M
\in
C^{\infty}(M).
\end{equation}

\begin{lemma}
\label{lma:approximate-metrics}
Let $(M,g)$ be a simple $C^{1,1}$ manifold. Let~$h$ be a smooth reference metric on~$M$. There exists a sequence $\left(\alf{g}\right)_{\alpha \in \N}$ of~$C^{\infty}$-smooth metrics on~$M$ such that
\begin{enumerate}[(1)]
\setlength\itemsep{0.2em}
\item \label{metrics1} $\alf{g}_{ij} \to g_{ij}$ in~$W^{2,2}_h(M)$ and in~$W^{1,\infty}_h(M)$,
\item \label{metrics2} $\alf{g}^{ij} \to g^{ij}$  in~$W^{1,2}_h(M)$\NTR{Fixed indices.} and in~$L^\infty_h(M)$,
\item \label{metrics3} $\alf{\Gamma}^i_{\ jk} \to \Gamma^i_{\ jk}$ in~$W^{1,1}_h(M)$ and in~$L^\infty_h(M)$,
\item \label{metrics4} $\alf{R}^i_{\ jkl} \to R^i_{\ jkl}$ in~$L^1_h(M)$.
\end{enumerate}
\end{lemma}

\begin{proof}
For each~$\alpha \in \N$ let $\alf{g} \coloneqq \alf{g}_{ij}\d x^i \otimes \d x^j \in C^{\infty}(T^2M)$, where~$\alf{g}_{ij}$ are as in~\eqref{eq:sequence-of-approximating-metrics}.
We will show that a subsequence of the sequence~$\left(\alf{g}\right)_{\alpha \in \N}$ consists of smooth Riemannian metrics and satisfies conditions~\ref{metrics1}--\ref{metrics4}.

We see that for large~$\alpha$ each~$\alf{g}$ is a~$C^\infty$-smooth Riemannian metric.
Smoothness follows standard properties of the mollifiers~$\alf{\varphi}$.
Each~$\alf{g}$ is symmetric by construction.
For large~$\alpha$ each~$\alf{g}$ is positive definite since this is an open condition and pointwise convergence~$\alf{g}_{ij} \to g_{ij}$ follows from continuity and item~\ref{metrics1}.

\ref{metrics1}
Because $\overline{g}_{ij}$ is compactly supported and in both spaces~$W^{2,2}(\R^n)$ and~$W^{2,\infty}(\R^n)$, the convolution converges $\alf{\varphi}\,\ast\,\overline{g}_{ij} \to \overline{g}_{ij}$ in both spaces~$W^{2,2}(\R^n)$ and~$W^{1,\infty}(\R^n)$.
This implies convergence\NTR{Fixed typo.} in the corresponding function spaces over the subdomain $M\subset\R^n$.


\ref{metrics2}
Let us denote the adjugate of a matrix~$A$ by~$\adj(A)$; we interpret rank two tensor fields as matrix-valued functions on $M\subset\R^n$.
By item~\ref{metrics1} we have
\begin{equation}
\label{eq:det-adj-supp}
\det(\alf{g}) \to \det(g)
\quad
\text{and}
\quad
\adj(\alf{g})^{ij} \to \adj(g)^{ij}
\end{equation}
in~$L^\infty_h(M)$.
Thus for sufficiently large~$\alpha$ the matrices~$\alf{g}$ are uniformly invertible in the sense that 
\begin{equation}
\norm{\det(\alf{g})^{-1}}_{L^\infty_h(M)}
\le
C
.
\end{equation}
Since
\begin{equation}
\alf{g}^{ij}(x)
=
\det(\alf{g}(x))^{-1}
\adj(\alf{g}(x))^{ij},
\end{equation}
we have that $\alf{g}^{ij}\to g^{ij}$ in~$L^\infty_h(M)$.
Derivatives of the inverse satisfy $\partial_k\alf{g}^{ij} = -\alf{g}^{il}(\partial_k\alf{g}_{lm})\alf{g}^{mj}$, which implies convergence of the derivatives in~$L^2_h(M)$.


\ref{metrics3}
This follows from\NTR{Fixed indices in the following formula.}
\begin{equation}
\alf{\Gamma}^i_{\ jk}
=
\frac12\alf{g}^{il}
\left(
\partial_j\alf{g}_{kl}
+
\partial_k\alf{g}_{jl}
-
\partial_l\alf{g}_{jk}
\right)
\end{equation}
and items~\ref{metrics1} and~\ref{metrics2}.

\ref{metrics4}
This follows from
\begin{equation}
\alf{R}_{ijk}^{\ \ \ l}
=
\partial_i\alf{\Gamma}^l_{\ jk}
-
\partial_j\alf{\Gamma}^m_{\ jk}
+
\alf{\Gamma}^m_{\ jk}\alf{\Gamma}^l_{\ im}
-
\alf{\Gamma}^m_{\ ik}\alf{\Gamma}^l_{\ jm}
\end{equation}
and item~\ref{metrics3}.
\end{proof}

Next we prove the Pestov identity for~$C^{1,1}$ regular metrics. In the context of lemma~\ref{lma:c11-pestov} the manifold~$(M,g)$ is simple~$C^{1,1}$, the Riemannian metric~$g$ is~$C^{1,1}$ regular, the function~$u$ is in~$\Lip_0(SM)$ and satisfies~$Xu \in H^1(SM)$ and~$\grad{v}u \in H^1(N,X)$.

\begin{proof}[Proof of lemma~\ref{lma:c11-pestov}]
Choose a smooth reference Riemannian metric~$h$ on~$M$ and let~$\left(\alf{g}\right)_{\alpha \in \N}$ be a sequence of smooth metrics approximating~$g$ as in lemma~\ref{lma:approximate-metrics}.
For each~$\alpha \in \N$ denote $\alf{u} \coloneqq u \circ \alf{r}$. Then by remark \ref{rem:regularity-preserved} we have~$\alf{u} \in H^1_0(\alf{S}M)$,~$\alf{X}\alf{u} \in H^1(\alf{S}M)$ and~$\alfGrad{v}\alf{u} \in H^1(\alf{N},\alf{X})$, which implies that~$\alf{u} \in \SSob(\alf{S}M)$ and~$\alf{u}|_{\partial(\alf{S}M)} = 0$.
For each~$\alpha$ an application of lemma~\ref{lma:u-sobolev-pestov} gives
\begin{equation}
\label{eq:pestov-alfa}
\norm{\alfGrad{v}\alf{X}\alf{u}}^2_{L^2(\alf{N})}
=
Q_{\alpha}\left(\alfGrad{v}\alf{u}\right)
+
(n-1)
\norm{\alf{X}\alf{u}}^2_{L^2(\alf{S}M)}.
\end{equation}
We will show that
\begin{equation}
\label{eq:alfa-gradv-x}
\lim_{\alpha \to \infty}
\norm{\alfGrad{v}\alf{X}\alf{u}}^2_{L^2(\alf{N})}
=
\norm{\grad{v}Xu}^2_{L^2(N)}.
\end{equation}
Similar arguments can be used to deduce that
\begin{equation}
\label{eq:alpha-q}
\lim_{\alpha \to \infty}
Q_{\alpha}\left(\alfGrad{v}\alf{u}\right)
=
Q\left(\grad{v}u\right)
\end{equation}
and
\begin{equation}
\label{eq:alpha-x}
\lim_{\alpha \to \infty}
\norm{\alf{X}\alf{u}}^2_{L^2(\alf{S}M)}
=
\norm{Xu}^2_{L^2(SM)}
.
\end{equation}
Then letting $\alpha \to \infty$ in equation~\eqref{eq:pestov-alfa} proves the claim of the theorem.
Since the arguments showing equations~\eqref{eq:alpha-q} and~\eqref{eq:alpha-x} are analogous to what is presented below, we omit them.
(The fact that components of the curvature tensor only converge in~$L^1$ and not in~$L^\infty$ is where the assumption $u\in\Lip(SM)$ is useful.)
Coordinate formulas required to show equations~\eqref{eq:alpha-q} and~\eqref{eq:alpha-x} are given in appendix~\ref{app:coord-forms-l2-norms}.

For any~$L^p$ convergence to make sense, we fix~$S_hM$ to be our common reference bundle for objects to be integrated on.
First we study how the~$L^2(\alf{N})$ norm on the left-hand side of~\eqref{eq:alfa-gradv-x} transforms under~$\alf{s}$.
Let~$\tilde u \coloneqq \alf{u} \circ \alf{s}^{-1}$ and fix~$\hat z \in S_hM$ and~$z \in \alf{S}M$ such that~$\alf{s}(z) = \hat z$.
By basic properties of pushforwards we have
\begin{equation}
\left(\left((\alf{s}_{\ast}\alf{X})\tilde u\right) \circ \alf{s}\right)(z)
=
(\alf{s}_{\ast}\alf{X})_{\hat z}\tilde u
=
\alf{X}_{z}(\tilde u \circ \alf{s})
=
\alf{X}_{z}\alf{u}.
\end{equation}
Thus
\begin{equation}
(\alf{s}_{\ast}\alf{\partial}^j)_{\hat z}
\left(
(\alf{s}_{\ast}\alf{X})\tilde u
\right)
=
\alf{\partial}^j_{z}
\left(
(\alf{s}_{\ast}\alf{X})\tilde u \circ \alf{s}
\right)
=
\alf{\partial}^j_{z}
\left(
\alf{X}\alf{u}
\right).
\end{equation}
Since~$\pi(z) = \pi(\hat z) \in M$ we get
\begin{equation}
\begin{split}
\norm{\alfGrad{v}\alf{X}\alf{u}}^2_{L^2(\alf{N})}
&=
\int_{z \in \alf{S}M}
\alf{g}_{ij}(\pi(z))
\left(
\alf{\partial}^i_{z}(\alf{X}\alf{u})
\right)
\left(
\alf{\partial}^j_{z}(\alf{X}\alf{u})
\right)
\,\d\alf{\Sigma}(z)
\\
&=
\int_{\hat z \in S_hM}
\alf{g}_{ij}(\pi(\hat z))
\left(
\alf{\partial}^i_{\alf{s}^{-1}(\hat z)}(\alf{X}\alf{u})
\right)
\\
&\qquad\qquad
\times
\left(
\alf{\partial}^j_{\alf{s}^{-1}(\hat z)}(\alf{X}\alf{u})
\right)
\abs{
\det
\left(
\d \alf{s}^{-1}_{\hat z}
\right)
}
\,\d\Sigma_h(\hat z)
\\
&=
\int_{\hat z \in S_hM}
\alf{g}_{ij}(\pi(\hat z))
\left(
(\alf{s}_{\ast}\alf{\partial}^i)_{\hat z}
\left(
(\alf{s}_{\ast}\alf{X})\tilde u
\right)
\right)
\\
&\qquad\qquad
\times
\left(
(\alf{s}_{\ast}\alf{\partial}^j)_{\hat z}
\left(
(\alf{s}_{\ast}\alf{X})\tilde u
\right)
\right)
\abs{
\det
\left(
\d \alf{s}^{-1}_{\hat z}
\right)
}
\,\d\Sigma_h(\hat z).
\end{split}
\end{equation}
An analogous formula holds for the right-hand side of equation~\eqref{eq:alfa-gradv-x}. Note that~$\tilde u = u \circ s^{-1}$. Thus we see that to prove equation~\eqref{eq:alfa-gradv-x} we need to prove the following two items.
\begin{enumerate}[(i)]
    \item\label{proof-of-pestov1}
    $
    (\alf{s}_{\ast}\alf{\partial}^j)
    \left(
    (\alf{s}_{\ast}\alf{X})\tilde u
    \right)
    \to 
    (s_{\ast}\partial^j)
    \left(
    (s_{\ast}X)\tilde u
    \right)
    $
    in $L^\infty(S_hM)$.
    \item\label{proof-of-pestov2}
    $
    \det
    \left(
    \d \alf{s}^{-1}
    \right)
    \to
    \det
    \left(
    \d s^{-1}
    \right)
    $
    in $L^\infty(S_hM)$.
\end{enumerate}

The push-forward~$\alf{s}_{\ast}$ has a useful block matrix representation in the coordinates of the unit sphere bundles.
Let~$(x,\alf{w}) \in \alf{S}M$ and~$(x,v) \in S_hM$ correspond to each other through~$\alf{s}(x,\alf{w}) = (x,v)$.
To~$(x,\alf{w})$ and~$(x,v)$ we associate the coordinate vector fields $\partial_{x^1},\dots,\partial_{x^n},\partial_{\alf{w}^1},\dots,\partial_{\alf{w}^n}$ and $\partial_{x^1},\dots,\partial_{x^n},\partial_{v^1},\dots,\partial_{v^n}$ respectively.
The matrix representation of~$\alf{s}_{\ast}$ in a block form with respect to the bases
$\partial_{x^1},\dots,\partial_{x^n},\partial_{\alf{w}^1},\dots,\partial_{\alf{w}^n}$ and
$\partial_{x^1},\dots,\partial_{x^n},\partial_{v^1},\dots,\partial_{v^n}$ is
\begin{equation}
\alf{s}_{\ast}
=
\begin{pmatrix}
I & 0 \\
\partial_xv & \partial_{\alf{w}}v
\end{pmatrix}.
\end{equation}
To find coordinate expressions for~$\alf{s}_{\ast}\alf{\partial}^j$ and~$\alf{s}_{\ast}\alf{X}$, the vector fields~$\alf{\partial}^j$ and~$\alf{X}$ need to be expressed in the basis $\partial_{x^1},\dots,\partial_{x^n},\partial_{\alf{w}^1},\dots,\partial_{\alf{w}^n}$.
As long as everything is only evaluated on~$\alf{S}M$, we have~$\alf{X}\alf{u} = \alf{w}^j\alf{\delta}_j\hat{u}$ for any~$\hat{u} \colon TM \setminus 0 \to \R$ such that~$\hat{u}|_{\alf{S}M} = \alf{u}$.
Therefore, in local coordinates and as long as we are careful to only evaluate only~$\alf{S}M$, we have
\begin{align}
\alf{X}
=
\alf{w}^j\partial_{x^j}
-\alf{\Gamma}^k_{\ jl}\alf{w}^j\alf{w}^l\partial_{\alf{w}^k}.
\end{align}
Similarly we get $\alf{\partial}^j = \alf{g}^{jl}\partial_{\alf{w}^l}$.

Coordinate formulas for the push-forwards vector fields can be found by multiplying with~$\alf{s}_{\ast}$.
This time only evaluating on $S_hM$, we have
\begin{equation}
\alf{s}_{\ast}\alf{X}
=
\alf{w}^j\partial_{x^j}
+
\left(
\alf{w}^k\partial_{x^k}v^j
-
\alf{\Gamma}^k_{\ lm}\alf{w}^l\alf{w}^m(\partial_{\alf{w}^k}v^j)
\right)
\partial_{v^j}
\end{equation}
and
\begin{equation}
\alf{s}_{\ast}\alf{\partial}^j
=
\alf{g}^{jl}
(\partial_{\alf{w}^l}v^k)
\partial_{v^k}.
\end{equation}
From these expressions it is clear that convergence in item~\ref{proof-of-pestov1} comes down to three matters.
In the base there are derivatives of components of~$\alf{g}$ up to the second order and derivatives of components of~$\alf{g}^{-1}$ up to the first order.
Again, components of the metric~$\alf{g}$ appear in the coefficients~$\alf{w}^j$.
The behaviour on the limit of all of these matters is controlled by lemma~\ref{lma:approximate-metrics}. We have concluded item~\ref{proof-of-pestov1}.

To prove item~\ref{proof-of-pestov2}, we write the matrix~$\d \alf{s}^{-1}$ in the block form
\begin{equation}
\d \alf{s}^{-1}
=
\begin{pmatrix}
I & 0 \\
\partial_{x}\alf{w} & \partial_{v}\alf{w}
\end{pmatrix}.
\end{equation}
Clearly~$\det\left(\d \alf{s}^{-1}\right) = \det\left(\partial_v\alf{w}\right)$.
Therefore the behaviour as~$\alpha \to \infty$ depends on sums and products of derivatives~$\partial_{v^k}(v\abs{v}^{-1}_{\alpha})$, which comes down to the metric~$\alf{g}$.
By lemma~\ref{lma:approximate-metrics} we have
$
\det\left(\d \alf{s}^{-1}\right)
\to
\det\left(\d s^{-1}\right)
$
in~$L^\infty(S_hM)$.

We have shown both items~\ref{proof-of-pestov1} and~\ref{proof-of-pestov2} and thus we have proved equation~\eqref{eq:alfa-gradv-x}.
This finishes the proof of the lemma\NTR{Switched 'theorem' to 'lemma'.}.
\end{proof}

\begin{lemma}
\label{lma:c11-commutator-formula}
Let~$(M,g)$ be a simple~$C^{1,1}$ manifold.
The commutator formula $[X,\grad{v}] = -\grad{h}$ holds on~$(M,g)$ in the~$H^1$ sense:  For all~$u \in H^1_0(SM)$ and~$V \in C^1(N)$ we have
\begin{equation}
\label{eq:c11-commutator}
\iip{\grad{h}u}{V}_{L^2(N)}
=
\iip{\grad{v}u}{XV}_{L^2(N)}
-
\iip{Xu}{\dive{v}V}_{L^2(SM)}.
\end{equation}
\end{lemma}

\begin{proof}
Let~$h$ be a smooth reference metric on~$M$ and choose a sequence~$\left(\alf{g}\right)_{\alpha \in \N}$ of smooth metrics approximating~$g$ as in lemma~\ref{lma:approximate-metrics}.
For each~$\alpha \in \N$ denote~$\alf{u} \coloneqq u \circ \alf{r}$ and~$\alf{V} \coloneqq V \circ \alf{r}$.
Then by remark~\ref{rem:regularity-preserved} we have~$\alf{u} \in H^1_0(\alf{S}M)$ and~$\alf{V} \in C^1(\alf{N})$.
We apply lemma~\ref{lma:cinfty-commutator} to~$\alf{u}$ and~$\alf{V}$ to get
\begin{equation}
\label{eq:alpha-comm}
\iip{\alfGrad{h}\alf{u}}{\alf{V}}_{L^2(\alf{N})}
=
\iip{\alfGrad{v}\alf{u}}{\alf{X}\alf{V}}_{L^2(\alf{N})}
-
\iip{\alf{X}\alf{u}}{\alfDive{v}\alf{V}}_{L^2(\alf{S}M)}.
\end{equation}
Letting~$\alpha \to \infty$ in equation~\eqref{eq:alpha-comm} proves the claimed identity~\eqref{eq:c11-commutator}, after we have shown that
\begin{align}
\label{eq:proof-of-comm1}
\lim_{\alpha \to \infty}
\iip{\alfGrad{h}\alf{u}}{\alf{V}}_{L^2(\alf{N})}
&=
\iip{\grad{h}u}{V}_{L^2(N)},
\\
\label{eq:proof-of-comm2}
\lim_{\alpha \to \infty}
\iip{\alfGrad{v}\alf{u}}{\alf{X}\alf{V}}_{L^2(\alf{N})}
&=
\iip{\grad{v}u}{XV}_{L^2(N)}
\end{align}
and
\begin{equation}
\label{eq:proof-of-comm3}
\lim_{\alpha \to \infty}
\iip{\alf{X}\alf{u}}{\alfDive{v}\alf{V}}_{L^2(\alf{S}M)}
=
\iip{Xu}{\dive{v}V}_{L^2(SM)}.
\end{equation}
All formulas~\eqref{eq:proof-of-comm1},~\eqref{eq:proof-of-comm2} and~\eqref{eq:proof-of-comm3} can be shown by arguments analogous to those used in proving the formula~\eqref{eq:alfa-gradv-x} in the proof of lemma~\ref{lma:c11-pestov}.
Thus we omit the details.
Coordinate formulas needed to complete the proofs of the formulas are given in appendix~\ref{app:coord-forms-l2-norms}.
\end{proof}

\subsection{Regularity of the integral function}

Let~$(M,g)$ be a simple~$C^{1,1}$ manifold.
In this section we will prove lemma~\ref{lemma:regularity-of-integral-functions} concerning regularity properties of the integral functions of Lipschitz functions and one-forms.
We prove a Lipschitz property for the geodesic flow in lemma~\ref{lma:flow-lip}.
The Lipschitz property lets us prove that the integral functions are Lipschitz in lemma~\ref{lma:uf-lip-for-bnd-vanishing}.
To prove~$H^1(N,X)$ regularity for the vertical gradients of the integral functions in lemma~\ref{lma:gradvuf-h1x}, we use the commutator formula from lemma~\ref{lma:c11-commutator-formula}.

On a compact manifold~$M$ and its unit sphere bundle~$SM$ all reasonable notions of distance are bi-Lipschitz equivalent.
Since in this section~$M$ will be the Euclidean closed ball, we choose Euclidean distances. 

\begin{lemma}
\label{lma:flow-lip}
Let~$(M,g)$ be a simple~$C^{1,1}$ manifold.
For~$z \in SM$ let~$[\tau_-(z),\tau_+(z)]$ be the maximal interval of existence of the geodesic~$\gamma_z$.
The geodesic flow~$\phi_t$ is Lipschitz continuous in~$z \in SM$:
There is a uniform~$L > 0$ so that for all $z,\hat z\in SM$ and $t\in[\tau_-(z),\tau_+(z)]\cap[\tau_-(\hat z),\tau_+(\hat z)]$ we have
\begin{equation}
d_{SM}(\phi_t(z),\phi_t(\hat z)) \le Ld_{SM}(z,\hat z).
\end{equation}
\end{lemma}

\begin{proof}
Let~$z, \hat z \in SM$.
Note that both lifted geodesics~$t \mapsto \phi_t(z)$ and~$t \mapsto \phi_t(\hat z)$ satisfy the equation~$X(\psi(t)) = \dot{\psi}(t)$ on the interval $[\tau_-(z),\tau_+(z)] \cap [\tau_-(\hat z),\tau_+(\hat z)]$.
Since the distance~$d_{SM}$ can be taken to be Euclidean and the Christoffel symbols of the metric~$g \in C^{1,1}(T^2M)$ are Lipschitz continuous, we get by Grönwall's inequality that
\begin{equation}
d_{SM}(\phi_t(z),\phi_t(\hat z))
\le
e^{K\abs{t-0}}d_{SM}(\phi_0(z),\phi_0(\hat z))
=
e^{K\abs{t}}d_{SM}(z,\hat z),
\end{equation}
for some~$K > 0$ independent of~$t$,~$z$ and~$\hat z$.
Since~$M$ is simple~$C^{1,1}$, the function~$\tau$ is uniformly bounded on~$SM$.
Thus we find a constant~$L > 0$ independent of~$z$ and~$\hat z$ such that~$e^{K\abs{t}} \le L$ uniformly for $t\in [\tau_-(z),\tau_+(z)] \cap [\tau_-(\hat z),\tau_+(\hat z)]$, which finishes the proof.
\end{proof}

\begin{lemma}
\label{lma:uf-lip-for-bnd-vanishing}
Let~$(M,g)$ be a simple~$C^{1,1}$ manifold. Let~$f \in \Lip_0(SM)$ and let~$u^f$ be the integral function of~$f$ defined by~\eqref{eq:uf}. Then~$u^f \in \Lip(SM)$.
\NTR{As one referee suggests, it is possible to prove a similar result assuming only $f|_{\partial_0SM}=0$. This, however, is quite complicated and there are additional technical difficulties cause by the low regularity. This is more relevant for tensor tomography, so it will be covered on a follow-up paper on that topic. The proof is too long to fit here reasonably, given the focus on scalar tomography. See also a later footnote.}
\end{lemma}

\begin{proof}
Let~$z,\hat z \in SM$ be so that~$\tau(\hat z) \leq \tau(z)$.
Then by a simple calculation\NTR{Fixed $w\leadsto\hat z$.}
\begin{equation}
\label{eq:uf-lip-estimate}
\begin{split}
\abs{u^f(z) - u^f(\hat z)}
&\le
(\tau(z) - \tau(\hat z))
\sup_{t \in [\tau(\hat z),\tau(z)]}
\abs{f(\phi_t(z))}
\\
&\quad+
\int_0^{\tau(\hat z)}
\abs{f(\phi_t(z)) - f(\phi_t(\hat z))}
\,\d t.
\end{split}
\end{equation}
We will show that both summands on the right-hand side of equation~\eqref{eq:uf-lip-estimate} are bounded by~$Cd_{SM}(z,\hat z)$ for some constant~$C > 0$ independent of~$z$ and~$\hat z$.

First, we treat the second term on the right-hand side of~\eqref{eq:uf-lip-estimate}.
Since by lemma~\ref{lma:flow-lip} the geodesic flow~$\phi_t$ and~$f$ both are Lipschitz, there is a constant~$K > 0$ independent of~$t$,~$z$ and~$\hat z$ so that
\begin{equation}
\abs{f(\phi_t(z)) - f(\phi_t(\hat z))} \le Kd_{SM}(z,\hat z).
\end{equation}
Since the manifold~$M$ is simple~$C^{1,1}$, there is a constant~$L > 0$ independent of~$\hat z$ so that~$\tau(\hat z) \le L$.
It follows that
\begin{equation}
\label{eq:conclusion1}
\int_{0}^{\tau(\hat z)}
\abs{f(\phi_t(z)) - f(\phi_t(\hat z))}
\,\d t
\le
KLd_{SM}(z,\hat z),
\end{equation}
which proves the desired bound for the second term.

Then we turn to the first term on the right-hand side of~\eqref{eq:uf-lip-estimate}.
Since~$f$ is Lipschitz and vanishes on the boundary~$\partial(SM)$, for all~$t \in [\tau(\hat z),\tau(z)]$ we have
\begin{equation}
\label{eq:bnd-distance}
\begin{split}
\abs{f(\phi_t(z))}
&=
\abs{f(\phi_t(z)) - f(\phi_{\tau(z)}(z))}
\\&\le
\Lip(f)d_{SM}(\phi_t(z),\phi_{\tau(z)}(z))
\\&\leq
\Lip(f)(\tau(z)-t)
\\&\leq
\Lip(f)(\tau(z)-\tau(\hat z))
\\&\leq
\Lip(f)(\tau(z)+\tau(\hat z)).
\end{split}
\end{equation}
The function~$\tau^2$ is Lipschitz since the manifold is simple~$C^{1,1}$, and so
\begin{equation}
\label{eq:conclusion2}
\begin{split}
(\tau(z) - \tau(\hat z))
\sup_{t \in [\tau(\hat z),\tau(z)]}
\abs{f(\phi_t(z))}
&\le
\Lip(f)(\tau^2(z) - \tau^2(\hat z))
\\
&\le
\Lip(f)\Lip(\tau^2)d_{SM}(z,\hat z),
\end{split}
\end{equation}
as desired.

Combining estimates~\eqref{eq:uf-lip-estimate},~\eqref{eq:conclusion1} and~\eqref{eq:conclusion2} yields a Lipschitz estimate for the integral function~$u^f$.
\end{proof}

\begin{lemma}
\label{lma:gradvuf-h1x}
Let~$(M,g)$ be a simple~$C^{1,1}$ manifold.
Assume that~$f \in\Lip_0(SM)$ integrates to zero over all maximal geodesics in~$M$.
Then~$\grad{v}u^f \in H^1(N,X)$, where~$u^f$ is the integral function of~$f$ defined by equation~\eqref{eq:uf}.
\end{lemma}

\begin{proof}
The integral function~$u^f$ is in~$\Lip(SM)$ by lemma~\ref{lma:uf-lip-for-bnd-vanishing} and~$u^f|_{\partial(SM)} = 0$ since~$f$ integrates to zero over all maximal geodesics of~$M$.
Thus by remark~\ref{rem:regularity-preserved}
we have~$u^f \in H^1_0(SM)$.
Then an application of lemma~\ref{lma:c11-commutator-formula} gives
\begin{equation}
\label{eq:commutator-application}
\iip{\grad{v}u^f}{XV}_{L^2(N)}
=
\iip{\grad{h}u^f}{V}_{L^2(N)}
-
\iip{Xu^f}{\dive{v}V}_{L^2(SM)}
\end{equation}
for any~$V \in C^1(N)$. Here~$Xu^f \in H^1(SM)$, since~$Xu^f = -f \in \Lip(SM)$.
As $Xu^f=-f=0$ at~$\partial SM$, for any~$V \in C^1(N)$\NTR{Zero boundary values for $V$ is not required because $u^f$ vanishes there. Updated this sentence.} we can integrate by parts in~\eqref{eq:commutator-application} to get
\begin{equation}
\iip{\grad{v}u^f}{XV}_{L^2(N)}
=
\iip{(\grad{h}-\grad{v}X)u^f}{V}_{L^2(N)}.
\end{equation}
Therefore $X\grad{v}u^f = (\grad{v}X - \grad{h})u^f \in L^2(N)$, which shows that~$\grad{v}u^f \in H^1(N,X)$.
\end{proof}

\begin{lemma}
\label{lma:unique-geods-imply-bnd-determ}
Let~$(M,g)$ be a simple~$C^{1,1}$ manifold. Then for any~$x \in \partial M$ and~$v \in S_x(\partial M)$, there is a sequence of vectors~$v_k \in S_xM$ so that~$\tau(x,v_k) > 0$,~$v_k \to v$ and~$\tau(x,v_k) \to 0$ as~$k \to \infty$.
\end{lemma}

\begin{proof}
Let~$x \in \partial M$ and~$v \in S_x(\partial M)$.
Choose a~$C^1$ boundary curve~$\sigma$ defined on an interval~$I$ so that~$\sigma(0) = x$ and~$\dot\sigma(0) = v$.
Choose a sequence~$(x_k)$ of boundary points on~$\sigma(I)$ so that~$x_k \to x$.
For each~$k$ let~$v_k \in S_xM$ be the initial velocity of the unique geodesic~$\gamma_{k}$\NTR{Renamed $\gamma_{xx_k}\leadsto\gamma_k$ in this proof.} joining~$x$ to~$x_k$ in the interior of~$M$ --- the geodesic~$\gamma_{k}$ exists by simplicity.
Then~$\tau(x,v_k) > 0$ for each~$k$.
Since the lengths of the geodesics depend continuously on their end points, we get~$\tau(x,v_k) = l(\gamma_{k}) \to 0$ as~$k \to \infty$.

It remains to verify that~$v_k \to v$.\NTR{Separated the proof of the convergence $v_k\to v$ into a separate paragraph (below).}
The geodesic equation gives
\begin{equation}
\abs{
\ddot{\gamma}_k^i(t)
}
=
\abs{
\sum_{j,l}
\Gamma^i_{\ jl}(\gamma_k(t))
\dot\gamma_k^j(t)\dot\gamma_k^l(t)
}
\leq
n^2
\sup_x\abs{\Gamma^i_{\ jl}(x)}
\sup_t\abs{\dot\gamma_k(t)}^2
,
\end{equation}
where all norms are the Euclidean ones of the global coordinates and the suprema range over all coordinates.
Therefore $\abs{\ddot{\gamma}_k(t)}$ is uniformly bounded for all~$t$ and~$k$,
and so by Taylor approximation in the coordinates
\begin{equation}
x_k
=
\gamma_k(\tau_k)
=
x
+
\tau_k v_k
+
\Order(\tau_k^2)
,
\end{equation}
where the error term is uniform over~$k$.
Therefore (in local coordinates)
\begin{equation}
\begin{split}
v
&=
\dot\sigma(0)
\\&=
\lim_{k\to\infty}
\frac{x_k-x}{\tau_k}
\\&=
\lim_{k\to\infty}
\frac{\tau_k v_k+\Order(\tau_k^2)}{\tau_k}
\\&=
\lim_{k\to\infty}
v_k
\end{split}
\end{equation}
as claimed.
\end{proof}

\begin{lemma}
\label{lma:bnd-determ-for-scalars}
Let~$(M,g)$ be a simple~$C^{1,1}$ manifold.
Suppose that~$f \in \Lip(M)$ integrates to zero over all maximal geodesics of~$M$. Then~$f$ vanishes on the boundary~$\partial M$.
\end{lemma}

\begin{proof}
Let~$x \in \partial M$ be a boundary point.
Suppose that~$v \in S_x(\partial M)$.
By lemma~\ref{lma:unique-geods-imply-bnd-determ} there is a sequence of tangent vectors~$v_k \in S_xM$ so that~$\tau(x,v_k) > 0$,~$\tau(x,v_k) \to 0$ and~$v_k \to v$ when~$k \to \infty$.
Since integrals of~$f$ over all maximal geodesics vanish, the integral function~$u^f$ of~$f$ vanishes on the boundary~$\partial(SM)$.
As the lengths of the geodesics approach zero we get
\begin{equation}
f(x)
=
\lim_{k \to \infty}
\frac{1}{\tau(x,v_k)}
\int_0^{\tau(x,v_k)}
f(\gamma_{x,v_k}(t))
\,\d t
=
\lim_{k \to \infty}
\frac{1}{\tau(x,v_k)}
u^f(x,v_k)
=
0
\end{equation}
as claimed.
\end{proof}

If $f\in\Lip(SM)$, then the proof above only gives $f|_{\partial_0(SM)}=0$, not $f|_{\partial(SM)}=0$.
This is true also in the smooth case, and this conclusion is optimal for general functions on~$SM$.\NTR{Added sentence.}
If a function on~$SM$ arises from a tensor field, then the natural boundary determination is more involved in low regularity and we shall not discuss it here; cf. remark~\ref{rmk:h=0}.\NTR{Added a sentence and removed an old one. See an earlier comment on tensor boundary determination.}

\begin{proof}[Proof of lemma~\ref{lemma:regularity-of-integral-functions}]
\ref{reg-of-uf}
Let~$f$ be a Lipschitz function on~$M$ that integrates to zero over all maximal geodesics of~$M$.
Define the integral function~$u^f$ of~$f$ as in~\eqref{eq:uf}.
We have~$f \in \Lip_0(M)$ by lemma~\ref{lma:bnd-determ-for-scalars}.
Thus~$u^f \in \Lip_0(SM)$ by lemma~\ref{lma:uf-lip-for-bnd-vanishing}.
We have~$\grad{v}u^f \in H^1_0(N,X)$ by lemma~\ref{lma:gradvuf-h1x} since~$u^f$ vanishes on~$\partial(SM)$. The last claim $Xu^f = -\pi^{\ast}f \in \Lip(SM) \subseteq H^1(SM)$ follows from the fundamental theorem of calculus.

\ref{reg-of-uh}
Let~$h$ be a Lipschitz~$1$-form on~$M$ that integrates to zero over all maximal geodesics of~$M$ and vanishes on the boundary~$\partial M$.
Let~$u^h$ be the integral function of~$h$ defined by~\eqref{eq:uf}.
Then~$u^h \in \Lip_0(SM)$ by lemma~\ref{lma:uf-lip-for-bnd-vanishing}.
We see that~$Xu^f \in H^1(SM)$ and~$\grad{v}u^h \in H^1_0(N,X)$ as in item~\ref{reg-of-uf}.
\end{proof}

\subsection{The integral function in the Pestov identity}

This subsection concludes the proofs of the lemmas required to prove theorem~\ref{thm:c11-injectivity}.
We verify that the integral function of a Lipschitz~$1$-form~$h$ on~$M$ behaves in the same way in the Pestov identity as it does in the smooth case.
Recall that~$(M,g)$ is a simple~$C^{1,1}$ Riemannian manifold and particularly~$g$ is a~$C^{1,1}$ regular Riemannian metric on~$M$.

\begin{proof}[Proof of lemma~\ref{lma:oneform-in-pestov-cancels}]
Let~$h$ be a Lipschitz~$1$-form on~$M$ and denote by~$\tilde h$ the associated function on~$SM$.
We will show that
\begin{equation}
\label{eq:one-forms-cancel}
\norm{\grad{v}\tilde h}_{L^2(N)}^2
=
(n-1)
\norm{\tilde h}_{L^2(SM)}^2.
\end{equation}
The Lipschitz assumption guarantees that the left-hand side of~\eqref{eq:one-forms-cancel} is well defined.
Let~$\omega$ stand for the $(n-1)$-dimensional measure of the unit sphere in~$\R^n$. By~\cite[Lemma 4]{IlmavirtaXRTPR} we have
\begin{equation}
\int_{S_xM}
\abs{\tilde h(x,v)}^2
\,\d S_x
=
\abs{h(x)}^2
\frac{\omega}{n}
\end{equation}
and
\begin{equation}
\int_{S_xM}
\abs{\grad{v}\tilde h(x,v)}^2
\,\d S_x
=
\abs{h(x)}^2
\frac{\omega(n-1)}{n}
\end{equation}
on every fiber~$S_xM$ of the unit sphere bundle. We may integrate over~$x$ just as in~\cite[Lemma 4]{IlmavirtaXRTPR} despite having less regularity, and we find
\begin{equation}
\norm{\grad{v}\tilde h}^2_{L^2(N)}
=
(n-1)
\int_M
\abs{h(x)}^2\frac{\omega}{n}
\,\d V_g
=
(n-1)\norm{\tilde h}^2_{L^2(SM)}
\end{equation}
as claimed.
\end{proof}

\section{Lemmas in smooth geometry}
\label{sec:lemmas-in-smooth-geometry}

This final section contains the proofs of the lemmas used to verify that the two definitions of simplicity (definitions~\ref{def:simple-a} and~\ref{def:simple-b}) agree when the geometry is~$C^\infty$-smooth.
We assume that~$M \subseteq \R^n$ is the closed unit ball and we let~$g$ be a~$C^{\infty}$-smooth Riemannian metric on~$M$.

We denote by~$I_{\gamma}$ the index form along a geodesic~$\gamma$ of~$M$.
Recall that if there are interior conjugate points along~$\gamma$, then~$I_{\gamma}$ is indefinite and if the end points of~$\gamma$ are conjugate to each other along~$\gamma$, then there is a normal vector field~$V \not\equiv 0$ along~$\gamma$ so that~$I_{\gamma}(V) = 0$.
If~$V$ is a normal vector field along~$\gamma$ vanishing at the end points of~$\gamma$, we abbreviate~$I_{\gamma}(V) \coloneqq I_{\gamma}(V,V)$.

\begin{proof}[Proof of lemma~\ref{lma:posit-q-implies-no-conj}]
Let~$(M,g)$ be a simple~$C^{1,1}$ manifold and assume that the Riemannian metric~$g$ is~$C^{\infty}$-smooth.
Let~$\gamma_0 \colon [a,b] \to M$ be a maximal geodesic in~$M$ and let~$V \not\equiv 0$ be a normal vector field along~$\gamma_0$ vanishing at the end points of~$\gamma_0$.
We will show that $I_{\gamma_0}(V) > 0$, proving that there cannot be conjugate points along~$\gamma$ even at its end points.

Let $(\gamma_0(0),\dot\gamma_0(0)) =: z_0 \in \dooin(SM)$ be the initial data of a geodesic~$\gamma_0$ and let~$\tilde\gamma_0$ be the lift to the sphere bundle.
The pullback bundle~$\tilde\gamma_0^{\ast}N$ consists precisely of all normal vector fields along~$\gamma_0$.
Particularly,~$V$ is a section of~$\tilde\gamma_0^{\ast}N$ vanishing at the end points, so by lemma~\ref{lma:extension} (in appendix~\ref{app:extension}) there is a smooth section~$\tilde V$ of~$N$ vanishing on the boundary and satisfying~$\tilde V|_{\tilde\gamma_0} = V$.
We may assume that~$\tilde V$ is supported in a small neighborhood of~$\tilde\gamma_0$.

Choose for each~$k \in \N$ a smooth function~$a_k \colon \dooin(SM) \to \R$ so that~$a_k^2 \to \delta_{z_0}$ in the weak sense and~$\int_{\dooin(SM)}a^2_k\d\mu = 1$, where $\d\mu(x,v) = \ip{\nu(x)}{v}\d\Sigma_g(x,v)$.
Since we are working locally around~$z_0$, it is enough to find such a sequence of functions in Euclidean space and we see that a sequence of square roots of positive standard mollifiers will\NTR{Fixed typo.} suffice.

For each~$k \in \N$ let~$W_k \in C^{\infty}(N)$ be a section such that~$W_k(\phi_t(z)) = a_k(z)\tilde V(\phi_t(z))$ for all~$z \in \dooin(SM)$ and~$t \in [0,\tau(z)]$.
By Santaló's formula (see~\cite[Lemma 3.3.2]{SharafutdinovRTRM}) it follows that as~$k \to \infty$ we have
\begin{equation}
\begin{split}
Q(W_k) 
&=
\int_{z \in \dooin(SM)}
I_{\gamma_z}(W_k|_{\tilde\gamma_z})
\,\d\mu(z)
\\
&=
\int_{z \in \dooin(SM)}
a^2_k(z)I_{\gamma_z}(\tilde V|_{\tilde\gamma_z})
\,\d\mu(z)
\\
&\to
\int_{z \in \dooin(SM)}
\delta_{z_0}(z)I_{\gamma_{z}}(\tilde V|_{\tilde\gamma_z})
\,\d\mu(z)
\\
&=
I_{\gamma_{0}}(\tilde V|_{\tilde\gamma_{0}})
=
I_{\gamma_{0}}(V).
\end{split}
\end{equation}
Here we have written the distribution~$\delta_{z_0}$ as a function on~$SM$ to simplify notation.
Similarly as~$k \to \infty$ we get
\begin{equation}
\begin{split}
\norm{W_k}^2_{L^2(N)}
&=
\int_{z \in \dooin(SM)}
\int_0^{\tau(z)}
\abs{W_k|_{\tilde\gamma_{z}}}^2
\,\d t
\,\d\mu(z)
\\
&=
\int_{z \in \dooin(SM)}
a^2_k(z)
\left(
\int_0^{\tau(z)}
\abs{\tilde V|_{\tilde\gamma_{z}}}^2
\,\d t
\right)
\,\d\mu(z)
\\
&\to
\int_{z \in \dooin(SM)}
\delta_{z_0}(z)
\left(
\int_0^{\tau(z)}
\abs{\tilde V|_{\tilde\gamma_{z}}}^2
\,\d t
\right)
\,\d\mu(z)
\\
&=
\int_0^{\tau(z_0)}
\abs{\tilde V|_{\tilde\gamma_{0}}}^2
\,\d t
=
\int_0^{\tau(z_0)}
\abs{V}^2
\,\d t.
\end{split}
\end{equation}
By $C^{1,1}$ simplicity of $(M,g)$ and zero boundary values of~$W_k$ there is $\varepsilon > 0$ so that $Q(W_k) \ge \varepsilon\norm{W_k}^2_{L^2(N)}$ for all~$k$. We conclude that
\begin{equation}
I_{\gamma_{0}}(V)
\ge
\varepsilon
\int_0^{\tau(z_0)}
\abs{V}^2
\,\d t
> 0,
\end{equation}
which proves that there cannot be conjugate points along~$\gamma_0$ even at its end points.
\end{proof}

\begin{proof}[Proof of lemma~\ref{lma:conv-bnd-equiv-tau2-lip}]
Let~$(M,g)$ be a compact smooth Riemannian manifold with a smooth boundary.
We assume that the Riemannian metric~$g$ is~$C^\infty$-smooth.

First, we will prove that strict convexity implies Lipschitz continuity of~$\tau^2$. As the boundary is strictly convex, all geodesics starting in the interior~$\sisus(SM)$ meet the boundary transversally. The implicit function theorem implies that~$\tau$ is smooth in~$\sisus(SM)$. As~$\tau \colon SM \to \R$ is continuous on all of~$SM$, it suffices to show that the gradient of~$\tau^2$ (in Sasaki or any other Riemannian metric on~$SM$) is uniformly bounded in the interior.

Let~$z \in SM$ be an interior point and let~$s \mapsto z_s$ be a smooth curve of interior points, where~$s \in (-\varepsilon,\varepsilon)$ and~$z_0 = z$. Choose~$s \mapsto z_s$ to have unit speed with respect to the Sasaki metric related to the $C^\infty$-smooth metric~$g$. The implicit function theorem gives an explicit formula for the differential~$\d \tau$ of~$\tau$. Applying the implicit function theorem to~$\rho(\gamma_{z_s}(t))$ yields
\begin{align}
\label{eq:dtau-ds}
\frac{\d}{\d s}\tau(z_s)
&=
-\frac{
\ip{
\frac{\d}{\d s}
\gamma_{z_s}(t)
}{
\nu(\gamma_{z_s}(t))}
}{
\ip{
\dot\gamma_{z_s}(t)
}{
\nu(\gamma_{z_s}(t))
}}
\bigg|_{t = \tau(z_s)}
,
\end{align}
where~$\rho$ is a boundary defining function.
To prove that~$\d(\tau^2) = 2\tau\d\tau$ is uniformly bounded in the interior, we will show that
\begin{equation}
\label{eq:tau-d-tau}
\tau(z_s)\frac{\d}{\d s}\tau(z_s)
\end{equation}
is bounded by some absolute constant near~$s = 0$.
Boundedness of~\eqref{eq:tau-d-tau} will follow after we have shown that\footnote{This estimate follows from \cite[Lemma 4.1.2]{SharafutdinovIGTF}, but we reprove it here. Our method of proof is different and may be of interest to some readers.}\NTR{Added this footnote. We find the proof in Sharafutdinov's book quite different, so we prefer to include our proof here instead of replacing it with a citation.}
\begin{equation}
\label{eq:tau-less-than-ip}
\tau(z)
\lesssim
\abs{
\ip{
\dot\gamma_{z}(\tau(z))
}{
\nu(\gamma_{z}(\tau(z)))
}}
\end{equation}
for all~$z \in \sisus(SM)$, since by growth estimates for Jacobi fields and $\abs{\dot z_0}=1$ we have
\begin{equation}
\abs{
\ip{
\frac{\d}{\d s}\gamma_{z_s}(\tau(z_s))
}{
\nu(\gamma_{z_s}(\tau(z_s)))
}}
\leq
C,
\end{equation}
where~$C$ is a constant depending only on curvature bounds and diameter.
Since the right-hand side of~\eqref{eq:tau-less-than-ip} is constant along the geodesic~$\gamma_z$, it is enough to prove boundedness for~$z \in \dooin(SM)$.

Outside any neighbourhood of the compact set~$\partial_0(SM)$, the right-hand side of~\eqref{eq:tau-less-than-ip} is uniformly bounded from below by a positive constant and~$\tau(z)$ is also uniformly bounded from above.
Thus if we can prove that there is a neighbourhood of the set~$\partial_0(SM)$ where~\eqref{eq:tau-less-than-ip} holds, it will hold everywhere on~$\partial(SM)$.

Take any~$x \in \partial M$ and an inward pointing vector~$v \in S_xM$.
Let
\begin{equation}
\hat x \coloneqq \gamma_{x,v}(\tau(x,v))
\quad
\text{and}
\quad
\hat v \coloneqq -\dot\gamma_{x,v}(\tau(x,v)).
\end{equation}
Let~$\nu$ be the inward unit normal vector at the boundary.
We decompose the vector~$\hat v$ 
as $\hat v^\perp\nu+\hat v^\parallel$, where $\hat v^\perp>0$ and $\hat v^\parallel$ is parallel to the boundary.\NTR{Completely rewrote this sentence for clarity.}
It follows from~\cite[Lemma 12]{IlmavirtaBRTBR} that as~$\hat v^\perp \to 0$, we have
\begin{equation}
\tau(\hat x,\hat v)
=
2v^{\perp}S(\hat v^{\parallel},\hat v^{\parallel})^{-1}
+
\Order((\hat v^{\perp})^2),
\end{equation}
where~$S$ is the second fundamental form of~$\partial M$ and the error term is locally uniform. As the boundary is strictly convex, the second fundamental form is bounded uniformly from below by~$c > 0$. Thus as~$\hat v^\perp \to 0$ we get
\begin{equation}
\label{eq:tau-less-than-perp}
\tau(\hat x,\hat v)
\le 3c^{-1}\hat v^{\perp}
=
3c^{-1}\ip{\nu(\hat x)}{\hat v}.
\end{equation}
Therefore, since~$\tau(x,v) = \tau(\hat x,\hat v)$, as~$v^\perp \to 0$ we get
\begin{equation}
\tau(x,v)
=
\tau(\hat x,\hat v)
\lesssim
\abs{\ip{\nu(\hat x)}{\hat v}}
=
\abs{\ip{\nu(\gamma_{x,v}(\tau(x,v)))}{\dot\gamma_{x,v}(\tau(x,v))}}.
\end{equation}
This shows that~\eqref{eq:tau-less-than-ip} holds in a neighbourhood of the tangential point~$(x,v^\parallel)$. Thus estimate~\eqref{eq:tau-less-than-ip} holds in a neighbourhood of~$\partial_0(SM)$.

Next we turn to the opposite statement.
We assume that~$\tau^2$ is Lipschitz.
If the boundary were not to be strictly convex everywhere, there would be a~$v \in S_x(\partial M)$ so that~$S_x(v,v) \le 0$.
\NTR{Reformulated the proof slightly so that the Lipschitz property is assumed, not so that we try to prove its failure. This paragraph has been restructured and cut off from the following ones.}

\NTR{This paragraph is new.}
As~$\tau^2$ is Lipschitz, the function~$\tau$ itself is H\"older-continuous.
Because the continuous function~$\tau$ vanishes on ${\dooout SM\setminus\partial_0(SM)}$ (the geodesics stop immediately),
we have
\begin{equation}
\label{eq:tau-tangent=0}
\tau|_{\partial_0(SM)}
=
0
\end{equation}
as well.



We use boundary normal coordinates near the base point~$x \in \partial M$.
We construct a family~$(\gamma_h)_{h \in [0,1]}$ of geodesics as follows.
Parallel translate the vector~$v$ for time~$h$ along the geodesic starting normally inwards from~$x$.
Call this vector $v_h\in T_{x_h}M$.
Let~$\gamma_h$ be the geodesic with the initial data $\dot\gamma_h(0)=v_h$.
The geodesic~$\gamma_0$ (with initial direction $v_0=v$ at $x_0=x$) starts at the boundary and may, depending on the convexity of the boundary, be only defined at $t=0$.

As in~\cite[Eq.~(2)]{IlmavirtaBRTBR} we extend the second fundamental form in the boundary normal coordinates near $x$.
Denote~$S_h(t) \coloneqq S_{\gamma_h(t)}(\dot\gamma_h(t),\dot\gamma_h(t))$.
Since~$S_x(v,v) \le 0$ we have
\begin{equation}
S_h(t)
=
S_0(0) + \Order(h) + \Order(\abs{t})
\le
C(h + \abs{t}),
\end{equation}
for some~$C > 0$ when~$h$ and~$\abs{t}$ are small.
If~$z_h(t)$ is the distance from~$\gamma_h(t)$ to the boundary, we have~$z_h(0) = h$ and~$\dot z_h(0) = 0$.
By writing the geodesic equation in boundary normal coordinates (as in~\cite[Eq.~(8)]{IlmavirtaBRTBR}) we find that
\begin{equation}
\label{eq:dist-to-bnd-ode}
\ddot z_h(t)
=
-S_h(t)
\ge
-C(h + \abs{t}).
\end{equation}

The total length~$\tau_h$ of the geodesic~$\gamma_h$ can be divided into forward and backward parts, denoted respectively by~$\tau^+_h$ and~$\tau^-_h$.
We want to find estimates for~$\tau^+_h$ and~$\tau^-_h$ from below.

Let us first consider the case of positive time,~$t > 0$.
Integrating the estimate~\eqref{eq:dist-to-bnd-ode} leads to
\begin{equation}
z_h(t)
=
h
+
\int_0^t
\int_0^s
\ddot z_h(r)
\der r
\der s
\geq
h-\frac{C}{2}ht^2 - \frac{C}{6}t^3
=:
\hat z_h(t)
\end{equation}
for all $t>0$.
If we choose
$
A
\coloneqq
\min
\left(
\sqrt{\frac{2}{3C}},
\sqrt[3]{\frac{2}{C}}
\right)
$ and $
\hat\tau^+_h
\coloneqq
Ah^{1/3}
$
,
then for all~$t \in [0,\hat\tau^+_h]$ we have
\begin{equation}
\hat z_h(t)
\ge
h\left[
1 - \frac12Ch^{2/3}A^2 - \frac16CA^3
\right]
\\
\ge
\frac h3.
\end{equation}
Therefore~$z_h(t) \ge \hat z_h(t) > 0$ for~$t \in [0,\hat\tau^+_h]$.
This shows that~$\tau^+_h \ge \hat\tau^+_h$.

The case of negative time can be reduced to previous case by substituting~$t = -s$,~$s > 0$ and similarly we get~$\tau^-_h \ge Ah^{1/3}$.
\NTR{Rephrased the proof from this point on.}
Equation~\eqref{eq:tau-tangent=0} implies $\tau_0=0$, and this together with the Lipschitz continuity of~$\tau^2$ implies that there is $B>0$ so that~$\tau_h^2 \le Bh$.
As~$0 < h \ll 1$, this gives us
\begin{equation}
Bh
\ge
\tau_h^2
=
(\tau^+_h
+
\tau^-_h)^2
\ge
4A^2h^{2/3},
\end{equation}
which is impossible for small~$h$.
This is a contradiction so the boundary has to be strictly convex.
\end{proof}

\appendix

\section{Coordinate formulas and norms}
\label{app:coord-forms-l2-norms}

We have collected here the remaining formulas from proofs of lemmas~\ref{lma:c11-pestov} and~\ref{lma:c11-commutator-formula}.
In the context of the proof of lemma~\ref{lma:c11-pestov} following formulas hold.
The~$Q_{\alpha}$-term in identity~\eqref{eq:pestov-alfa} is
\begin{equation}
Q_{\alpha}\left(\alfGrad{v}\alf{u}\right)
=
\norm{\alf{X}\alfGrad{v}\alf{u}}_{L^2(\alf{N})}
-
\iip{\alf{R}\alfGrad{v}\alf{u}}{\alfGrad{v}\alf{u}}_{L^2(\alf{N})}.
\end{equation}
For~$L^2$ quantities in identity~\eqref{eq:pestov-alfa} we have
\begin{equation}
\label{eq:appendix1}
\begin{split}
\norm{\alf{X}\alfGrad{v}\alf{u}}^2_{L^2(\alf{N})}
&=
\int_{S_hM}
\alf{g}_{ij}
\left(
\alf{w}^k
\left(
\alf{s}_{\ast}\alf{\delta}_k
\right)
\left(
(\alf{s}_{\ast}\alf{\partial}^i)
\tilde u
\right)
+
\alf{\Gamma}^i_{\ lk}\alf{w}^l(\alf{s}_{\ast}\alf{\partial}^k)
\tilde u
\right)
\\
&\hspace{4em}\times
\left(
\alf{w}^k
\left(
\alf{s}_{\ast}\alf{\delta}_k
\right)
\left(
(\alf{s}_{\ast}\alf{\partial}^j)
\tilde u
\right)
+
\alf{\Gamma}^j_{\ lk}\alf{w}^l(\alf{s}_{\ast}\alf{\partial}^k)
\tilde u
\right)
\\
&\hspace{4em}\times
\abs{
\det
\left(
\d \alf{s}^{-1}
\right)
}
\,\d \Sigma_h
\end{split}
\end{equation}
and
\begin{equation}
\label{eq:appendix2}
\norm{\alf{X}\alf{u}}^2_{L^2(\alf{S}M)}
=
\int_{S_hM}
\abs{
\left(
\alf{s}_{\ast}\alf{X}
\right)
\tilde u}^2
\abs{
\det
\left(
\d \alf{s}^{-1}
\right)
}
\,\d \Sigma_h
\end{equation}
and
\begin{equation}
\label{eq:appendix3}
\begin{split}
\iip{\alf{R}\alfGrad{v}\alf{u}}{\alfGrad{v}\alf{u}}_{L^2(\alf{N})}
&=
\int_{S_hM}
\alf{g}_{ij}
\left(
\alf{R}^i_{\ jkl}
\left(
(\alf{s}_{\ast}\alf{\partial}^j)\tilde u
\right)
\alf{w}^k\alf{w}^l
\right)
\\
&\hspace{4em}\times
\left(
(\alf{s}_{\ast}\alf{\partial}^j)\tilde u
\right)
\abs{
\det
\left(
\d \alf{s}^{-1}
\right)
}
\,\d\Sigma_h.
\end{split}
\end{equation}
For the vector fields~$\alf{s}_{\ast}\alf{\delta}_k$,~$\alf{s}_{\ast}\alf{X}$ and~$\alf{s}_{\ast}\alf{\partial}^j$ appearing in formulas~\eqref{eq:appendix1},~\eqref{eq:appendix2} and~\eqref{eq:appendix3} we have coordinate formulas
\begin{align}
\alf{s}_{\ast}\alf{\delta}_k
&=
\partial_{x^k}
+
(\partial_{x^k}v^j)\partial_{v^j}
-
\alf{\Gamma}^i_{\ kj}\alf{w}^j(\partial_{\alf{w}^i}v^l)
\partial_{v^l},
\\
\alf{s}_{\ast}\alf{X}
&=
\alf{w}^j\partial_{x^j}
+
\left(
\alf{w}^k\partial_{x^k}v^j
-
\alf{\Gamma}^k_{\ lm}\alf{w}^l\alf{w}^m(\partial_{\alf{w}^k}v^j)
\right)
\partial_{v^j}
\quad\text{and}
\\
\alf{s}_{\ast}\alf{\partial}^j
&=
\alf{g}^{jl}
(\partial_{\alf{w}^l}v^k)
\partial_{v^k}.
\end{align}

In the context of the proof of lemma~\ref{lma:c11-commutator-formula} the following formulas hold.
For the~$L^2$ inner products in equation~\eqref{eq:alpha-comm} we have
\begin{equation}
\label{eq:appendix4}
\iip{\alfGrad{h}\alf{u}}{\alf{V}}_{L^2(\alf{N})}
=
\int_{S_hM}
\alf{g}_{ij}
\left(
(\alf{s}_{\ast}\alf{\delta}^i)\tilde u
+
\alf{w}^i(\alf{s}_{\ast}\alf{X})\tilde u
\right)
\tilde V^j
\abs{
\det
\left(
\d \alf{s}^{-1}
\right)
}
\,\d\Sigma_h
\end{equation}
and
\begin{equation}
\label{eq:appendix5}
\begin{split}
\iip{\alfGrad{v}\alf{u}}{\alf{X}\alf{V}}_{L^2(\alf{N})}
&=
\int_{S_hM}
\alf{g}_{ij}
\left(
(\alf{s}_{\ast}\alf{\partial}^i)\tilde u
\right)
\left(
(\alf{s}_{\ast}\alf{X})\tilde V^j
+
\alf{\Gamma}^j_{\ lk}\alf{w}^l\tilde V^k
\right)
\\
&\hspace{4em}\times
\abs{
\det
\left(
\d \alf{s}^{-1}
\right)
}
\,\d\Sigma_h
\end{split}
\end{equation}
and
\begin{equation}
\label{eq:appendix6}
\iip{\alf{X}\alf{u}}{\alfDive{v}\alf{V}}_{L^2(\alf{S}M)}
=
\int_{S_hM}
\left(
(\alf{s}_{\ast}\alf{X})\tilde u
\right)
\left(
(\alf{s}_{\ast}\alf{\partial}_j)\tilde V^j
\right)
\abs{
\det
\left(
\d \alf{s}^{-1}
\right)
}
\,\d\Sigma_h.
\end{equation}
New vector fields~$\alf{s}_{\ast}\alf{\partial}_j$ and~$\alf{s}_{\ast}\alf{\delta}^k$ appear in equations~\eqref{eq:appendix4},~\eqref{eq:appendix5} and~\eqref{eq:appendix6}. For them we have the coordinate formulas
\begin{equation}
\alf{s}_{\ast}\alf{\partial}_j
=
(\partial_{\alf{w}^j}v^k)\partial_{v^k},
\end{equation}
and
\begin{equation}
\alf{s}_{\ast}\alf{\delta}^k
=
\alf{g}^{kl}\partial_{x^l}
+
(\alf{g}^{kl}(\partial_{x^l}v^j))\partial_{v^j}
-
\alf{g}^{kl}\alf{\Gamma}^i_{\ lm}\alf{w}^m(\partial_{\alf{w}^i}v^j)\partial_{v^j}.
\end{equation}

\section{Smooth extension from a curve}
\label{app:extension}

This appendix is devoted to the proof of the following lemma.
We will comment on some of the definitions and give examples after the statement.
In this appendix everything is smooth and all manifolds and bundles have finite dimension.

\begin{lemma}
\label{lma:extension}
Let~$M$ be a smooth manifold with boundary and $\pi\colon B\to M$ a bundle over it whose fiber is a closed manifold.
Let $\Pi\colon E\to B$ be a vector bundle over~$B$.

Let $\sigma\colon[a,b]\to B$ be a smooth curve\NTR{This need not be a section of any bundle. This can be any smooth curve satisfying the given conditions.} without self-intersections so that the end points~$\pi(\sigma(a))$ and~$\pi(\sigma(b))$ are on~$\partial M$ and $\pi(\sigma(t))\in\sisus(M)$ for all $t\in(a,b)$.
Suppose the exit directions~$\dot\sigma(a)$ and~$\dot\sigma(b)$ are not tangent to the boundary $\partial B \coloneqq \pi^{-1}(\partial M)$. 

Let~$V$ be a smooth section of the pullback bundle~$\sigma^*E$ so that $V(a)=V(b)=0$.
Then there is a smooth section~$W$ of~$E$ so that $W|_{\partial B}=0$\NTR{Fixed typo in the restriction.} and $W(\sigma(t))=V(t)$ for all $t\in [a,b]$.
\end{lemma}

The fiber of the bundle~$B$ is a smooth and compact manifold of any finite dimension, including zero.
The result is valid in the trivial case where the fiber is a singleton and $B=M$.
If~$E$ is the trivial line bundle $B\times\R$, then sections of it are merely scalar functions $B\to\R$.
Therefore the lemma covers extensions of scalar functions from smooth curves~$\gamma$ on~$M$ but also much more.
The result will only be applied in the case $B=SM$ and $E=N$, but we record it in more generality as it adds no cost.

As $\sigma\colon[a,b]\to B$ is an injective smooth map, a section of the pullback bundle~$\sigma^*E$ is simply a smooth map $W\colon[a,b]\to E$ so that $\Pi(W(t))=\sigma(t)$ for all $t\in[a,b]$.

\begin{proof}[Proof of lemma~\ref{lma:extension}]
Denote the projected curve by $\gamma\coloneqq\pi\circ\sigma\colon[a,b]\to M$.
The assumption that~$\dot\sigma(a)$ and~$\dot\sigma(b)$ are not tangential to~$\partial B$ implies that the\NTR{Fixed typo.} end directions~$\dot\gamma(a)$ and~$\dot\gamma(b)$ on the base are not tangential to~$\partial M$.

The point $x=\gamma(a)$ has a neighborhood $\omega_1\subset M$\NTR{Fixed typo. $\omega$s live on $M$ and $\Omega$s on $B$.} where we may choose local coordinates $\phi\colon\omega_1\to\R^n$ so that $\phi(\partial M\cap\omega_1)=\{x_n=0\}$ and for all interior points $y\in M\setminus\partial M$ we have $\phi_n(y)>0$.
In these coordinates the initial direction satisfies $\dot\gamma_n(a)>0$, and so the map
\begin{equation}
\theta
\colon
[a,a+\eps)
\ni
t
\mapsto
\gamma_n(t)
\in
[0,h)
\end{equation}
is a diffeomorphism for some choice of $\eps,h>0$.

We may shrink~$\omega_1$ so that $\phi(\omega_1)\subset\R^{n-1}\times[0,h)$ and the bundle $B$ is locally trivial:
$
B
\supset
\pi^{-1}(\omega_1)
\approx
\omega_1\times F
$%
,
where $F$ is a closed manifold (the typical fiber of~$B$).
Denote $y=\sigma(a)\in B_x=F$.
There is a neighborhood $U\ni y$ in~$F$ so that the bundle~$E$ is trivial over $\omega_1\times U\eqqcolon\Omega_1\subset\pi^{-1}(\omega_1)\subset B$ (with the product in the sense of the local trivialization of~$B$)\NTR{Redescribed $\Omega_1$ more clearly.} in the sense that $\Pi^{-1}(\Omega_1)\approx\Omega_1\times\R^K$, where $K\in\N$ is the dimension of the fiber of~$E$.
In these coordinates the section~$V$ of~$\sigma^*E$ may be written as a smooth function $[a,b]\to\R^K$\NTR{Fixed typo in the interval.}, and we denote the component functions as $V_k\colon[a,b]\to\R$\NTR{Fixed typo in the interval.}.
By the non-intersecting property of~$\sigma$ we may assume the neighborhoods $\omega_1\subset M$ and $\Omega_1\subset B$ to be so small that the curve~$\sigma$ does not return to~$\Omega_1$ after leaving it.

We define a function $W_1\colon\Omega_1\to\R^K$ by letting its components be
\begin{equation}
\label{eq:W-ext-def}
W_1(z)_k
=
V_k(\theta^{-1}(\pi(z)_n))
.
\end{equation}
This defines a section~$W_1$ of the bundle~$E$ in a neighborhood of the point $(x,y)\in B$.
By construction $W_1(z)=0$ when $z\in\partial B$, as that corresponds to the set where $\pi(z)_n=0$ and we have $V(a)=0$.
This section~$W_1$ satisfies the required restriction property where it is defined:
Whenever $t\in[a,b]$ satisfies $\sigma(t)\in\Omega_1$, we have $W_1(\sigma(t))=V(t)$.

Similarly, there is a neighborhood~$\Omega_2$ of $(\gamma(b),\sigma(b))\in B$ and a local section $W_2\colon\Omega_2\to E$ with the same property:
Whenever $t\in[a,b]$ satisfies $\sigma(t)\in\Omega_2$, we have $W_2(\sigma(t))=V(t)$.

In addition to satisfying the restriction property, both of these local sections~$W_1$ and~$W_2$ of~$E$ vanish on the boundary~$\partial B$ when defined there.
The point of the construction in~\eqref{eq:W-ext-def} is to ensure that the local extension vanishes on the boundary.

For any $t\in(a,b)$ it is easy to provide local extensions as~$\sigma$ has no self-intersections and there are no boundary conditions to worry about.
Using compactness of $\sigma([a,b])$ to pass to a finite subcover, we find sets $\Omega_3,\dots,\Omega_J\subset B\setminus \partial B$ and local sections $W_j\colon\Omega_j\to E$ of~$E$ so that $W_j(\sigma(t))=V(t)$ whenever $\sigma(t)\in\Omega_j$ and $\sigma([a,b])\subset\bigcup_{j=1}^J\Omega_j$.

We also let $\Omega_0=B\setminus\sigma([a,b])$ and let $W_0\colon\Omega_0\to E$ be the zero section.
The vector field~$W_0$ has the same boundary conditions and restriction properties as the other~$W_j$s but for trivial reasons.

The sets $\Omega_0,\dots,\Omega_J$ are an open cover of the smooth manifold~$B$ with boundary~$\partial B$.
Let the functions $\psi_0,\dots,\psi_J\in C_c^\infty(B)$ be a partition of unity subordinate to this cover in the sense that each~$\psi_j$ is supported in~$\Omega_j$ and $\sum_{j=0}^J\psi_j(z)=1$ for all $z\in B$.
The functions $B\to E$ defined by $\psi_j(z)W_j(z)$ are smooth (interpreted to be zero outside~$\Omega_j$ where~$W_j$ is defined) and the global smooth section $W\colon B\to E$ given by
\begin{equation}
W(z)
=
\sum_{j=0}^J
\psi_j(z)W_j(z)
\end{equation}
is quickly verified to have all the required properties.
\end{proof}

\bibliographystyle{alpha}
\bibliography{references}

\end{document}